\documentclass[11pt,a4paper,reqno]{amsart}
\usepackage[latin1]{inputenc}
\usepackage[english]{babel}
\usepackage{amsmath}
\usepackage{amsfonts}
\usepackage{amssymb}
\usepackage{mathrsfs}
\usepackage{latexsym}
\usepackage{yfonts}
\usepackage{natbib}

\usepackage[margin=3cm]{geometry}

\usepackage{color}

\usepackage{graphicx,color}

\usepackage[plainpages=false,colorlinks,hyperindex,bookmarksopen,linkcolor=red,citecolor=blue,urlcolor=blue]{hyperref}

\bibpunct{[}{]}{;}{n}{,}{,}

\newtheorem{tm}{Theorem}[section]

\newtheorem{prop}{Proposition}[section]
\newtheorem{remark}{Remark}[section]

\numberwithin{equation}{section}

\allowdisplaybreaks

\title{Limit of p-Laplacian Obstacle problems}
	
	\author{Raffaela Capitanelli, Maria Agostina Vivaldi}
	\address{Department of Basic and Applied Sciences for Engineering\newline Sapienza University of Rome\newline via A. Scarpa 10, Rome, Italy\newline raffaela.capitanelli@uniroma1.it, maria.vivaldi@sbai.uniroma1.it}
\begin{document}

\date{\today}

\begin{abstract}
In this paper we study  asymptotic behavior of solutions to obstacle problems  for $p-$Laplacians as $p\to \infty.$
For the one-dimensional case and for the radial case, we give an explicit expression of the limit. In the n-dimensional case, we 
provide sufficient conditions   to assure the uniform convergence of  whole family of the solutions   of obstacle problems  either for  data $f$ that change sign in $\Omega$  or  for data $f$ (that do not change sign in $\Omega$)  possibly  vanishing  in a set of positive measure.

\end{abstract}

\maketitle


\bigskip AMS 35J60,  35J65,  35B30, 35J87

\section{Introduction}

The study of obstacle problems for both $p$-Laplacian and $\infty$-Laplacian, which  has  recently received a strong impulse, is closely connected with many relevant  topics  such as 
 the mass optimization problems, the Absolutely  Minimizing Lipschitz Extensions, the Infinity Harmonic Functions, the Monge-Kantorovich mass transfer problem and  the Tug of War Games. We mention for instance \cite{ALS}, \cite{ACJ},   \cite{BDM},  \cite{BDR},  \cite{BBD}, \cite{EG},  \cite{FKR},  \cite{J}, \cite{JPR}, \cite{LW}, \cite{MRT}, \cite{RTU},   \cite{VI}, and the references therein.

In this paper we study the asymptotic behavior of solutions to obstacle problems  for $p-$Laplacians as $p$ tends to $\infty.$
Let  $\Omega\subset \mathbb{R}^n$ denote a bounded domain. We consider the problem:
\begin{equation} \label{OPp}
\text{find } u\in \mathcal{K},\, \,\,\,\int_{\Omega} |\nabla u|^{p-2}\nabla u\nabla(v-u)  \,dx\, -\int_{\Omega} f(v-u)\,dx\,\geqslant 0\,\,\,
  \forall v\in \mathcal{K}, 
  \end{equation}
where $$\mathcal{K}=\{v\in W_0^{1,p}( \Omega):\ v\geqslant\varphi\,\text{in}\,\, \Omega\,\}$$
with obstacle $\varphi\in  W^{1,p}( \Omega),$    $\varphi\leq 0$ on  $\partial \Omega,$ and the datum
  \begin{equation} \label{f}f\in L^\infty(\Omega).\end{equation}
Then, for any fixed $p,$ there exists a unique solution $u_p.$
If we assume \begin{equation} \label{triangle}-\triangle_p \varphi \in L^{p'}(\Omega),\end{equation}  where $-\triangle_pu =-div(|\nabla u|^{p-2 }\nabla u),$ then  the following  Lewy-Stampacchia inequality  holds (see \cite{TR})
\begin{equation}\label{LS}   f\leq -\triangle_p u_p\leq   -\triangle_p  \varphi \vee f.\end{equation}

Moreover, see for instance \cite{MRT} and Theorem 3.1 in \cite{CF},  if  \begin{equation}\label{knonv} \mathcal{K}^{\infty}=\{u\in W_0^{1,\infty}(\Omega): u\geqslant\varphi\ \text{ in }\Omega,||\nabla u||_{L^{\infty}(\Omega)}\leq 1\}\neq \emptyset\end{equation}
 then the family of the solution $u_p$ is pre-compact in $C(\bar\Omega);$  in particular,  from any sequence $u_{p_k}$  we can extract a subsequence  $u_{{p_k}_j}$ converging to a function $u_{\infty}$  in  $C(\bar\Omega)$, $u_{\infty}$ being a maximizer of the following  problem:
\begin{equation}\label{eq:3DV}
\mathcal{F}(u)=\max\biggl\{ \mathcal{F}(w): w\in \mathcal{K}^{\infty}\biggl\}\end{equation}
where  $$\mathcal{F}(w)=\int_{\Omega} w(x)f(x) dx.$$
Moreover,  \begin{equation}\label{limsup}  \limsup_{p\to\infty} ||\nabla u_p||_{L^{\infty}(\Omega)}\leq 1. \end{equation}

The limit  Problem \eqref{eq:3DV} is related to an optimal mass transport problem with taxes. More precisely, in  \cite{MRT}  it is proved  that solutions to  obstacle problems  for $p-$Laplacians  give an approximation to the extra production/demand necessary  in the process and to a Kantorovich potential for the corresponding transport problem.  Moreover, in  \cite{MRT}, the authors also show that this problem can be interpreted as an optimal mass transport problem with courier.

In this paper we face the question whether the whole family of the solutions  $u_p$ of the obstacle Problem (\ref {OPp}) is convergent to the same limit function $u_\infty.$   
For the analogous results for Dirichlet problems we mention  \cite {BDM}, \cite {BBD},  \cite {EG}, \cite {FMc}, \cite {HL},  and the references therein. 
The  asymptotic behavior of minimizers of $p$-energy forms on fractals as the Sierpinski Gasket (as  $p\to \infty$) has been recently addressed in  \cite{CCV}.

In the present paper, we give an explicit expression of the limit for the one-dimensional case and for the radial case (see Theorems \ref{main} and \ref{radial}). For arbitrary 
  $n-$dimensional domains, we provide sufficient conditions to assure the uniform convergence of  whole family of the solutions  to obstacle problems  either for  data $f$ that change sign in $\Omega$  or  for data $f$ (that do not change sign in $\Omega$)  possibly  vanishing  in a set of positive measure (see Theorems \ref{tf}, \ref{tf3}, \ref{tf2}, \ref{tf4} and \ref{tf5}). 
 Our paper has been  deeply inspired by Ishii and Loreti,  \cite {HL}, nevertheless
 the obstacle problems present their own  peculiarities and structural difficulties. In  Remarks  \ref{nolim}, \ref{E2}, \ref{E1}, \ref{E4} and \ref{E5} we highlight some peculiarities.  The main difficulties are due to the  fact that the solution $u_p$ to Problem \ref{OPp} satisfies the equation only on the set where  it  is detached from the obstacle. As this set depends on  $p$ then we have to deal  with Dirichlet problems with non homogeneous boundary conditions  in intervals moving with $p$ (see Theorem \ref{TD}, Proposition \ref{PDp} and Remark \ref{spezzo}). Hence the behavior of coincidence sets $\Gamma _p$ (\ref{gamma}) plays a crucial role (see condition  \eqref{Lambda}). As the regularity properties of the free boundaries are important tools for the study of the behavior of coincidence sets, then our approach is strictly related to the papers \cite{RTU} and  \cite{BDR}. In particular Theorem 2.8 in \cite{RTU} as well as Theorems  7.5  and 1.3 in \cite{BDR} provide sufficient conditions to assure  that condition (\ref{Lambda}) holds. We note that in  \cite{RTU} and  \cite{BDR}   smoothness assumptions are required while in our paper we  deal with a larger class of obstacles and data.
In Section  \ref{example} we give examples of obstacle problems where condition (\ref{Lambda}) is satisfied even if  neither the assumptions of  Theorem 2.8 in \cite{RTU} nor  those of  Theorem 7.5 in \cite{BDR} are  satisfied. 
We note that hypothesis (\ref{Lambda}) is not assumed in Theorems  \ref{tf},  \ref{tf3},   \ref{tf2} and  \ref{tf4}.    In Theorem  \ref{tf}    concerning data $f$  changing sign in $\Omega,$ 
condition \eqref{f3} puts in relation  the position of the support  of $f$  with respect  to the  boundary of  $\Omega$ and it provides an alternative assumption that, in some sense, forces the coincidence sets  to have a  \emph{good} behavior.  Similarly the sign conditions  on the datum $f$  in Theorems    \ref{tf3} and   \ref{tf2} provide  alternative assumptions. Furthermore we remark that, as the constraint  in the convex $\mathcal{K}$ is from below, then  as a consequence of the Lewy-Stampacchia inequality \eqref{LS}, the  easy situation is when $f $  (possibly vanishing  in a set of positive measure) is non negative  while, when $f $ is non positive, we have to require also  conditions on $-\triangle_p\varphi$  (see \eqref{f115}  and  \eqref{f116} respectively). Finally, in Section   \ref{example} we give some  examples of  \emph{non trivial} obstacle
problems where all the assumptions of Theorem \ref{tf5} are satisfied (see  Examples 2 and 5).

As  mentioned above, our topic is   intrinsically related to the Absolutely Minimizing Lipschitz Extensions (AMLEs), to  viscosity solutions to the obstacle problem for the $\infty$-Laplacian and  to comparison principles for $\infty-$superharmonic functions (see \cite{J} and  \cite{RTU}).  To prove Theorems  \ref{tf},  \ref{tf3},   \ref{tf2}, \ref{tf4}  and  \ref{tf5} we  use such  approaches and tools.
 More precisely, under suitable assumptions,  every  sequence of solutions $u_p$ to   obstacle problems  \eqref{OPp},  being viscosity solutions (with respect to the $p$-Laplacian), converges to a viscosity solution $u_\infty$ of the obstacle problem for the $\infty$-Laplacian, which  is the smallest continuous  $\infty-$superharmonic function  above the obstacle.  Hence  the limit $u_\infty$  is unique. In fact, 
  among the solutions to  Problem (\ref{eq:3DV})  the limit  $u_\infty$ is the (unique)  Absolutely Minimizing Lipschitz Extension (AMLE) according to the terminology of \cite{ACJ} (see Example 6 in Section \ref{example}).
 In \cite{RTU} the authors consider obstacle problems for  both the $\infty$-Laplacian and  the $p$-Laplacians (see also \cite{BDR} for similar results). 
 Theorems \ref{tf},  \ref{tf3},   \ref{tf2}, \ref{tf4}, and  \ref{tf5} refer to   a more general class of problems and require  weaker   assumptions than  the ones in \cite{RTU} (see  Section 4). 
Moreover  Theorems \ref{main} and   \ref{radial}  provide, for the limit  of solutions $u_p,$ a simple representation in terms of the data. We note that the proofs of  Theorems \ref{main} and  \ref{radial}  do not involve the deep, delicate theory of  viscosity solutions for $\infty$-Laplacian  and AMLE solutions.

The plan of the paper is the following. Section 2 concerns  one-dimensional Dirichlet problems  with non homogeneous  boundary data, Section 3 concerns  the one-dimensional obstacle problem.  In Section 4 we consider the n-dimensional case. In the last section we provide   examples, comments and remarks.

\section{One-dimensional Dirichlet problem   with non homogeneous  boundary data }

We consider Dirichlet problems  with non homogeneous  boundary data in  the {one}-dimensional case.
More precisely, we consider  the following
 problem on $\Omega=(a,b),$ \begin{equation} \label{DP} 
\text{find } u\in \mathcal{K}_{D},\, \,\,\,\int_{\Omega} |\nabla u|^{p-2}\nabla u\nabla v   \,dx\, =\int_{\Omega} f  v \,dx\\,\,\,
  \forall v\in W_0^{1,p}( \Omega), 
  \end{equation}
where $$\mathcal{K}_{D}=\{u\in W^{1,p}( \Omega): u(a)=A_p, u(b)=B_p\}.$$
For any fixed $p,$  and $f\in L^\infty(\Omega) $ there exists a unique solution $u_p.$ 
By proceeding as in \cite{HL}, we can prove  that, if \begin{equation}\label{ipo}\frac{|B_p-A_p|}{b-a}\leq 1,\qquad  A_p\to A, \qquad B_p\to B, \end{equation} then $u_{p}\longrightarrow u_{\infty}$  weakly in $W^{1,m}(\Omega),\forall m>2$, $u_{\infty}$ being a maximizer of the following variational problem
\begin{equation}\label{eq:3DNO}
\int_{\Omega}u_{\infty}(x)f(x) dx=   \max  \biggl\{\mathcal{F}(w)  : w\in \mathcal{K}_D^{\infty}\biggl\}  
\end{equation} where   $$\mathcal{F}(w)=\int_{\Omega}w(x)f(x) dx$$ $$\mathcal{K}_D^{\infty}=\{u\in W^{1,\infty}(\Omega): u(a)=A, u(b)=B ,||\nabla u||_{L^{\infty}(\Omega)}\leq 1\}.$$

From now on we denote by $\mu(E)$ the $n$-dimensional Lebesgue measure  of the set $ E\subset \mathbb {R}^n.$

More precisely, the following theorem holds.
\begin{tm}\label{TD}
Suppose that  \eqref{f} and  \eqref{ipo} hold. Then  $u_{p}$ converges uniformly to the following function $U\in \mathcal{K}_D^{\infty}:$
\begin{equation}\label{solU} U(x)=\int_a^x (\chi_{O_-}-\chi_{O_+}+k \chi_{O_0}) dt +A\end{equation}
where 
$$O_-=\{x\in (a,b), F <\beta^* \}, \quad\quad O_+=\{x\in (a,b), F >\beta^* \},\quad\quad O_0=\{x\in (a,b), F =\beta^* \}$$
$$F(x)=\int_a^x  f(t)dt, \quad \quad h(r)=\mu(\{x\in \Omega: F(x)<r\})$$
\begin {equation} \label {star}
\beta^*=
\sup\{r\in \mathbb{R}: h(r)\leq\frac{b-a-A+B}2\}\,
\end {equation}
and 
 \begin {equation} \label {cappa}
k= \begin{cases}
   \frac{ \mu(O_+)  -\mu(O_-)  -A+B } {\mu(O_0)}   \qquad if \qquad  \mu(O_0)>0, \\
 0 \qquad if \qquad  \mu(O_0)=0.\\
\end{cases}
\end {equation}
 \end{tm}

We skip the proof, as it is similar to the proof of  following Theorem \ref{main}.
   
    \begin{remark} If \begin{equation}\label{ipo5}\frac{|B_p-A_p|}{b-a}\geq1\end{equation} the solution \eqref{solU} does not depend on the  datum  $f.$ 
    \end{remark}

More precisely, we have the following proposition.
   
   \begin{prop}
If \begin{equation}\label{ipo3}\frac{|B_p-A_p|}{b-a}\geq 1  ,\qquad  A_p\to A, \qquad B_p\to B \end{equation}then  \begin{equation}\label{U0} U(x)=A+ (\frac{B-A}{b-a} ) (x-a).\end{equation}

\end{prop}

  \begin{proof} 
  
First we consider $ \frac{|B_p-A_p|}{b-a}=1$ and  $A>B.$ Then \begin {equation} \beta^*=
\sup\{ r\in \mathbb{R},  h(r)\leq 0\}
\end {equation}
  and then $\beta^*=F_-=\min_{x\in [a,b]}  F(x).$
  So $O_-=\emptyset$ and if $\mu(O_o)>0$, then $k=-1$ (see \eqref{cappa}) and  \eqref{U0}  is proved.

If    $ \frac{|B_p-A_p|}{b-a}=1$ holds and $A<B,$ then  $\beta^*= \sup\{ r\in \mathbb{R}: h(r)\leq b-a\}=+\infty$ and then  $O_+=O_0=\emptyset$ and $O_-=(a,b)$  and    \eqref{U0}  is  showed.

  If   $ \frac{|B_p-A_p|}{b-a}=D_p>1 $    holds,  we consider $u_p = D_p  v_p$ where $v_p$ solve  \begin {equation} \label {ep}-\frac{d}{dx}(|u'(x)|^{p-2} u'(x))= \frac{f(x)}{D^{p-1}_p}\end {equation}  with $v_p(a)=A_p/D_p,$ $v_p(b)=B_p/D_p$  and
  \begin{equation}\frac{|B_p-A_p|}{D_p(b-a)}=1.\end{equation}
  Then $v_p$ converges to   $$V(x)=\frac 1D (A+ (\frac{B-A}{b-a} ) (x-a))$$ where  $D=\frac{|B-A|}{b-a}$  and then \eqref{U0}  is proved. 
    \end{proof}   
   
    We note that the result of Theorem \ref{TD}  holds also for a family of  Dirichlet problems in moving intervals. 
More precisely,  consider the problems on $\Omega_p=(a_p,b_p),$ $f\in L^\infty(\Omega)$
\begin{equation} \label{Dp}\begin {cases}
\text{find } u_p\in W^{1,p}( \Omega_p)\,\, \text{such that}\,\,  u_p(a_p)=A_p, \,\,u_p(b_p)=B_p, \text{and}\\
\int_{\Omega_p} |\nabla u|^{p-2}\nabla u\nabla v   \,dx\, =\int_{\Omega_p} f  v \,dx  \,\,\,\forall v\in W_0^{1,p}( \Omega_p). \end{cases}
  \end{equation}

Then the following proposition holds (we skip the proof, as it is similar to the proof of Theorem \ref{main}).
  \begin{prop} \label{PDp} Suppose
    \begin{equation}\label{ipop}\frac{|B_p-A_p|}{b_p-a_p}\leq 1,\,\, \text{and} \,\,A_p\to A ,\,\, B_p\to B ,\quad a_p\to a ,\,\, b_p\to b, \quad a_p\geq a,\,\, \quad b_p\leq b. \end{equation} 
  Then the solution $u_{p}$ converges  (locally) uniformly  in $(a,b)$ to the  function $U$ defined in (\ref{solU}).
\end{prop}
\begin{remark}  \label{spezzo}    From the previous proposition  we deduce  that  for any choice of family  of points  $x_p\in [a,b),\, x_p\to \eta \in [a,b)$  and  of points $y_p\in (a,b], \,\,y_p>x_p,\,\,\, y_p\to \gamma \in (a,b]$ with $\eta<\gamma$ the solutions  $v_p $ of  Problems (\ref{Dp})  in the intervals $(x_p, y_p)$ converges
 (locally) uniformly  in $(\eta,\gamma)$ to the restriction to the interval  $(\eta,\gamma)$ of the function $U$ defined in (\ref{solU}).   \end{remark}

  \section{One-dimensional Obstacle problem}

We  consider the   obstacle problem \eqref{OPp} on $\Omega=(a,b).$ 

We define the closed set 
\begin{equation}\label{gamma}
\Gamma_p=\{x\in\bar{\Omega}: u_p=\varphi \}; \end{equation} 
we set  $$\Gamma_\infty =\liminf \Gamma_p\quad \text{and }\qquad\Gamma^*_\infty =\limsup \Gamma_p,$$ 
and we recall that $$ \limsup \Gamma_p=\cap^\infty_{p=1}\cup_{n\geq p} \Gamma_n \,  \text{and } \liminf \Gamma_p=\cup^\infty_{p=1}\cap_{n\geq p} \Gamma_n $$ 
and we simply write  $\lim \Gamma_p$ if $\Gamma_\infty=\Gamma^*_\infty$  (for the definition of $\limsup \Gamma_p$ and of $\liminf \Gamma_p$ we refer to \cite {K}).

\begin{tm} \label{main}  We assume hypotheses  (\ref{f}), (\ref{triangle}), (\ref{knonv})  and 
\begin{equation}\label{Lambda}
\overline{int \Gamma_\infty^*}\subset \Gamma_\infty .\end{equation}
Then the solution $u_{p}$ converges uniformly to the following function $U\in \mathcal{K}^{\infty}:$
 $$ U=\varphi \,\, \text{in}\,\, \Gamma_\infty $$ and  for any (connected) component $(d, e),$  $[d, e]\subset(a,b)$ of $ \Omega \setminus \Gamma_\infty$  
 \begin {equation} \label {U}  U(x)=\int_d^x (\chi_{O_-}-\chi_{O_+}+k \chi_{O_0}) dt +\varphi(d)\end {equation} 
where 
$$O_-=\{x\in (d,e), F <\beta^* \}, \quad\quad O_+=\{x\in (d,e), F >\beta^* \},\quad\quad O_0=\{x\in (d,e), F =\beta^* \}$$

\begin {equation} \label {Fd}
F(x)=\int_d^x  f(t)dt, \quad \quad h(r)=\mu(\{x\in (d,e) : F(x)<r\})\end {equation}

\begin {equation} \label {star1}
\beta^*= \sup\{ r\in \mathbb{R}: h(r)\leq\frac{e-d-\varphi(d)+\varphi(e)}2\}\, 
\end {equation}
\begin {equation} \label {cappa1}
k= \begin{cases}
   \frac{ \mu(O_+)  -\mu(O_-)  -\varphi(d)+\varphi(e) } {\mu(O_0)}   &\qquad if \qquad  \mu(O_0)>0,\\
 0 &\qquad if \qquad  \mu(O_0)=0.\\
\end{cases}
\end {equation}

For any (connected) component $(a,c)$ of $ \Omega \setminus \Gamma_\infty$    
 \begin {equation} \label {U2}  U(x)=\int_a^x (\chi_{O_-}-\chi_{O_+}+k \chi_{O_0}) dt \end {equation} 
where 
$$O_-=\{x\in (a,c), F <\beta^* \}, \quad\quad O_+=\{x\in (a,c), F >\beta^* \},\quad\quad O_0=\{x\in (a,c), F =\beta^* \}$$
$$F(x)=\int_a^x  f(t)dt, \quad \quad h(r)=\mu(\{x\in (a,c) : F(x)<r\})$$
\begin {equation} \label {star2}
\beta^*= \sup\{ r\in \mathbb{R}: h(r)\leq\frac{c-a+\varphi(c)}2\}\,
\end {equation}
\begin {equation} \label {cappa2}
k= \begin{cases}
  \frac{ \mu(O_+)  -\mu(O_-)  +\varphi(c) } {\mu(O_0)}   \qquad &if \qquad  \mu(O_0)>0,  \\  
 0 \qquad &if \qquad  \mu(O_0)=0.\\
\end{cases}
\end {equation}

For any (connected) component $(d, b)$   of $ \Omega \setminus \Gamma_\infty$  
 \begin {equation} \label {U3}  U(x)=\int_b^x (\chi_{O_-}-\chi_{O_+}+k \chi_{O_0}) dt \end {equation} 
where 
$$O_-=\{x\in (d,b), F <\beta^* \}, \quad\quad O_+=\{x\in (d,b), F >\beta^* \},\quad\quad O_0=\{x\in (d,b), F =\beta^* \}$$
$$F(x)=\int_b^x  f(t)dt, \quad \quad h(r)=\mu(\{x\in (d,b) : F(x)<r\})$$
\begin {equation} \label {star3}
\beta^*= \sup\{ r\in \mathbb{R}: h(r)\leq\frac{b-d-\varphi(d)}2\}\, 
\end {equation}
\begin {equation} \label {cappa3}
k= \begin{cases}
   \frac{ \mu(O_+)  -\mu(O_-)  -\varphi(d) } {\mu(O_0)}   \qquad &if \qquad  \mu(O_0)>0,\\
 0 \qquad &if \qquad  \mu(O_0)=0.\\
\end{cases}
\end {equation}
 \end{tm}

From now on we denote by $ Lip_1(\bar\Omega)$  the space of the  Lipschitz functions with Lipschitz constant less or equal to $1.$

\begin{remark}
We note that if $\varphi\leq 0 $ on  $\partial \Omega,$ then the assumption   $\varphi\in Lip_1(\bar\Omega)$   implies that the convex $\mathcal{K}^\infty$ is not empty  but  this condition is not necessary. In fact,  on $\Omega=(-2,2),$ the obstacle  $\varphi= 1-x^2$ does not belong to the space $Lip_1(\bar\Omega)$ while assumption  (\ref{knonv}) is satisfied  as the following  function   $w$ belongs to $ \mathcal{K}^{\infty}$   $$w=\begin{cases}  0   \qquad & -2<x\leq -\frac54\\
x+\frac54 \qquad  & -\frac54<x\leq -\frac12\\
1-x^2\qquad &-\frac12<x\leq \frac12 \\
 -x+\frac54 \qquad  &\frac12<x\leq  \frac54\\
0  \qquad &\frac54<x\leq 2
\end{cases} $$
(see Section \ref{example}). 
 
 \end{remark}

 Before proving Theorem \ref{main},  we establish the following preliminary results which  take into account the three different cases for the connected components of $\Omega \setminus \Gamma_\infty.$

\begin{prop}\label{prop2} 
Let  $ x_p\in(a,b)$  and    $ y_p\in (x_p,b)$   such that $ -\triangle u_p=f$ in $(x_p, y_p),$ $u_p(x_p)=\varphi(x_p), $    $u_p(y_p)=\varphi(y_p).$
 If  \begin{equation}\label{up2}\frac{|u_p(y_p)-u_p(x_p)|}{y_p-x_p}\leq 1,\end{equation}
 then  there exists a unique value of $\beta$, say $\beta_p,$ such that 
  $$  u_p(y_p)= u_p(x_p)  +\int_{x_p}^{y_p} \psi_p(\beta_p -F_p^{**}(t))dt $$ 
  where $\psi_p(s)= |s|^{\frac1{p-1}-1}s$ for $s\in \mathbb{R}$ and  $F_p^{**}(x)=\int_{x_p}^x f(t) dt.$
 
 Moreover, \begin{equation}\label{Fb}\beta_p\in [F_{-,p}  -1, F_{+,p}  + 1] \quad\,\,\text{where}\quad\,\ F_{+,p}=\max_{[x_p,y_p] } F_p^{**}, \qquad  F_{-,p}=\min_{[x_p,y_p]}  F_p^{**}. \end{equation} 

\end{prop}

\begin{proof}
  
  We recall that  the solution $u_p$ belongs to  $C^1([a,b])$ (see (\ref{f}), (\ref{triangle}) and (\ref{LS})).
  
  According to  \cite{HL}, we obtain that for any $x\in (x_p, y_p)$ \begin{equation}\label{eqHL2}  u_p(x)= u_p(x_p)   +\int_{x_p} ^x \psi_p(\beta_p -F_p^{**}(t))dt \end{equation}
  with $\psi_p(s)= |s|^{\frac1{p-1}-1}s$ for $s\in \mathbb{R},$  $F_p^{**}(x)=\int_{x_p}^x f(t) dt$ and \begin{equation}\label{b}\beta_p=  |u'_p(x_p)|^{p-2} u'_p(x_p) .   \end{equation}
  
  By the property of $\psi_p,$ there exists a unique value of $\beta$, say $\beta_p,$ such that 
  $$  u_p(y_p)= u_p(x_p)  +\int_{x_p}^{y_p} \psi_p(\beta_p -F_p^{**}(t))dt .$$

  We observe that \begin{equation}\label{bb2}\beta_p\in [F_{-,p}  -1, F_{+,p}  + 1]\end{equation}
  where   $ F_{+,p}$ and  $F_{-,p}$ are defined in  (\ref{Fb}).

  We verify that  $$  \int_{x_p}^{y_p} \psi_p( F_{+,p}   +1 -F_p^{**}(t))dt  \geq     u_p(y_p)- u_p(x_p).$$
  If $  u_p(y_p)- u_p(x_p) \leq 0,$ the previous inequality holds trivially. Suppose $ u_p(y_p)- u_p(x_p) >0,$ then
  $$(F_{+,p}   +1 -F_p^{**}(t))\geq (\frac{u_p(y_p)- u_p(x_p)} {y_p-x_p})^{p-1}$$
where we	 use   \eqref{up2}.

Now we verify that  $$\int_{x_p}^{y_p} \psi_p( F_{-,p }  -1 -F_p^{**}(t))dt  \leq     u_p(y_p)- u_p(x_p).$$     If $u_p(y_p)-u_p(x_p) \geq 0,$ the previous inequality holds trivially. Suppose $u_p(y_p)-u_p(x_p) <0,$
then $$(-F_{-,p }   +1 +F_p^{**}(t) )\geq (\frac{u_p(x_p)- u_p(y_p)} {y_p-x_p})^{p-1}$$
where we	 use   \eqref{up2}.
\end{proof}

By proceeding as in the proof of Proposition \ref{prop2} we can show the following result that concerns the second case.

\begin{prop}\label{prop1} 

Let  $ x_p\in (a,b),$ such that $-\triangle u_p=f$ in $(a, x_p),$ $u_p(x_p)=\varphi(x_p),\,u_p(a)=0.$

If  \begin{equation}\label{up1}\frac{|u_p(x_p)-u_p(a)|}{x_p-a}\leq 1\end{equation}
 then  there exists a unique value of $\beta$, say $\beta_p,$ such that 
   $$  u_p(x_p)= \int_a^{x_p} \psi_p(\beta_p -F(t))dt $$
 where $\psi_p(s)= |s|^{\frac1{p-1}-1}s$ for $s\in \mathbb{R}$ and $F(x)=\int_a^x f(t) dt.$
 
 Moreover, $\beta_p\in [\hat{F}_--1, \hat{F}_+ +1]$ where 
 \begin{equation}\label{F+} \hat{F}_+=\min_{[a,x_p] } F \qquad  \hat{F}_-=\min_{[a,x_p]}  F. \end{equation} 
\end{prop}

 For the last case in Theorem \ref{main} we establish the following result that can be proved as Proposition \ref{prop2}.

\begin{prop}\label{prop3} 
Let  $ x_p\in (a,b)$   such that $-\triangle_pu_p=f$ in $(x_p, b),$ $u_p(x_p)=\varphi(x_p), $    $u_p(b)=0.$ If  \begin{equation}\label{up3}\frac{|u_p(b)-u_p(x_p)|}{b-x_p}\leq 1,\end{equation} then  there exists a unique value of $\beta$, say $\beta_p,$ such that 
  $$  u_p(x_p)= -\int_{x_p}^b \psi_p(\beta_p +F^*(t))dt$$
   where $\psi_p(s)= |s|^{\frac1{p-1}-1}s$ for $s\in \mathbb{R}$ and  $F^*(x)=\int_x^b f(t) dt.$ 
 
 Moreover, \begin{equation}\label{T+}\beta_p\in [T_- -1 , T_+ +1] \quad \text{where } T_{+}=\max_{[x_p,b] }(- F^*(x)), \quad  T_{-}=\min_{[x_p,b]} (- F^*(x)). \end{equation} 

\end{prop}
 
Now we prove Theorem  \ref{main}.
  
    \begin{proof}

  We split the proof in 4 steps.

  \textbf{Step 1.}
Let  the interval   $(d,e)$   is a  (connected) component of $ \Omega \setminus \Gamma_\infty$   such that $[d,e]\subset (a,b)$  and we assume that 
\begin{equation}  \label{step1}    
\begin{cases}
\text  {there exist } x_p>a,\,  \text  { and }    y_p\in (x_p,b) \,\,  \text  {such that} \,\,-\triangle_pu_p=f  \quad   in \quad (x_p, y_p),\\
u_p(x_p)=\varphi(x_p), \,  \,\,u_p(y_p)=\varphi(y_p), \,\,\, \ x_p\to d,\,\,    y_p\to e
\end{cases}
\end{equation}
 and \begin{equation}\label{up4}\frac{|u_p(y_p)-u_p(x_p)|}{y_p-x_p}\leq 1\end{equation}
By Proposition \ref{prop2}   there exists a unique value of $\beta$, say $\beta_p\in [F_{-,p}  -1, F_{+,p}  + 1],$  
such that 
  $$  u_p(y_p)= u_p(x_p)  +\int_{x_p}^{y_p} \psi_p(\beta_p -F_p^{**}(t))dt $$ 
  where $\psi_p(s)= |s|^{\frac1{p-1}-1}s$ for $s\in \mathbb{R}$ and  $F_p^{**}(x)=\int_{x_p}^x f(t) dt.$
 Moreover $F_{-,p}\to F_{-,d} ,$  $F_{+,p}\to F_{+,d},$  where 
 
 \begin{equation}\label{Fpd} F_{+,d}=\max_{[d,e] } F \qquad  F_{-,d}=\min_{[d,e]}  F \end{equation} 
 and $F$ is defined in (\ref{Fd}).
 We note that $F_{+,p}\leq F_+-F_- $ and $F_{-,p}\geq -F_++ F_-$ where \begin{equation}\label{F+a} F_+=\max_{[a,b] }  \int_a^x f(t)dt \qquad  F_-=\min_{[a,b]}  \int_a^x f(t)dt. \end{equation}

We set  $\delta(F)= 2(F_++1-F_-).$
  According to  Formula (iv) on page 419, Lemma 3.2 and 3.3 in  \cite{HL},  the following properties hold:

1 $\lim_{t\to r^-}h(t)= h(r)\leq \mu(\{x\in (d,e): F(x)\leq r\})=\lim_{t\to r^+}h(r),$ $F$ defined in (\ref{Fd});

2 $h(r)$ is strictly increasing in  $ [F_{-,d}, F_{+,d} ];$

 3 for $\beta\in [F_{-,p}  -1, F_{+,p}  + 1],$ $$ |   \psi_p(\beta -F_p^{**}(x))| \leq     \psi_p(\delta(F))\leq  \psi_1(\delta(F));$$

4 let $\alpha_j\in[F_{-,p}  -1, F_{+,p}  + 1]$ be a sequence converging to some $r\in \mathbb{R}$ and let $p_j$ be a sequence such that $p_j\to \infty.$ Then, for any $\phi\in L^1(\Omega),$
$$ \int_{O_-(r)}  \phi \,\psi_{p_j}(\alpha_j -F_{p_j}^{**}(x)) dx \to  \int_{O_-(r)}  \phi \, dx$$
$$ \int_{O_+(r)} \phi \,   \psi_{p_j}(\alpha_j -F_{p_j}^{**}(x)) dx \to  - \int_{O_+(r)}  \phi\, dx$$
  with 
$O_-(r)=\{x\in (d,e), F <r \}$
and 
$O_+(r)=\{x\in (d,e), F >r \}.$

In fact let $x\in O_-(r)$ then $r-F(x)=\delta_0>0$  there exist a positive constant $\delta$  and  an index $j_0$  such that  for any $j\geq j_0$ $$\delta \leq  \alpha_j -F_{p_j}^{**}(x)\leq \delta(F)$$ and so $$  \psi_{p_j}(\alpha_j -F_{p_j}^{**}(x))  \to  1.$$
By  property 3 and the  Lebesgue convergence we obtain the first limit.
If  $x\in O_+(r)$ then $r-F(x)= -\delta_0<0$  there exist a positive constant $\delta$  and  an index $j_0$  such that  for any $j\geq j_0$ $$-\delta(F) \leq  \alpha_j -F_{p_j}^{**}(x)\leq -\delta$$ and so $$  \psi_{p_j}(\alpha_j -F_{p_j}^{**}(x)) \to  1.$$
By  property 3 and the Lebesgue convergence we obtain the second limit. From \eqref{up4}, we deduce 
$$\frac{|\varphi (d)-\varphi (e)|}{e-d}\leq 1.$$
First we suppose that  
\begin {equation} \label{fimmag}\varphi (d)-\varphi (e)>d-e \end{equation} 
and we deduce that  $\beta^*\leq F_{+,d}.$  In fact if $ \beta^*> F_{+,d} $,  then $h(\beta^*)=e-d$  which  contradicts the inequality $h(\beta^*) \leq \frac {e-d-\varphi (d)+\varphi (e)}{2}< e-d$ that follows from  definition  (\ref{star1}).

Now we show that $$\lim_{p\to\infty}\beta_p=\beta^*.$$
First we prove that 
\begin {equation} \label{contra2} \liminf_{p\to\infty}\beta_p\geq\beta^*.
\end{equation}   
By contradiction we suppose that there exists  a sequence $p_j\to \infty$ such that  $\liminf_{p\to\infty}\beta_p=r<\beta^*.$
From the strictly monotonicity we have $$\lim_{t\to r^+}h(t)<h(\beta^*)\leq \frac{e-d-\varphi(d)+\varphi(e)}2.$$
Let $H=\{x\in (d,e), F \leq r\}$ and $L=\{x\in (d,e), F >r\}.$
Then we have $$\lim_{t\to r^+}h(t)=\mu(H)<\frac{e-d-\varphi(d)+\varphi(e)}2$$ (see property 1) and 
$$\mu(L)=e-d -\mu(H) >e-d  -\frac{e-d-\varphi(d) +\varphi(e)}2=\frac{e-d+\varphi(d)-\varphi(e)}2.$$
By property 3, we obtain $$\limsup_{j\to\infty} |\int_H\psi_{p_j}(\beta_{p_j} -F_{p_j}^{**}(x)) dx|\leq \limsup_{j\to\infty} \int_H \psi_{p_j}(\delta(F)) dx =\mu (H) <\frac{e-d-\varphi(d)+\varphi(e)}2.$$
By property 4, we obtain $$\lim_{j\to\infty} \int_L \psi_{p_j}(\beta_{p_j} -F_{p_j}^{**}(x)) dx= -\mu (L) < - \frac{e-d-\varphi(d)+\varphi(e)}2.$$
As $$u_{p_j}(y_{p_j})-u_{p_j}(x_{p_j})= \int_{x_{p_j}}^{y_{p_j}}\psi_{p_j} (\beta_{p_j} -F_{p_j}^{**}(x)) dx$$
passing to the limit for $j\to\infty$ we obtain 
$$\limsup_{j\to\infty} (u_{p_j}(y_{p_j})-u_{p_j}(x_{p_j}))< \varphi(e)-\varphi(d)$$ and that is a contradiction.
In fact $$\limsup_{j\to\infty} ( \varphi(e) -u_{p_j}(x_{p_j}))\leq \limsup_{j\to\infty}  (u_{p_j}(y_{p_j})-u_{p_j}(x_{p_j}))< \varphi(e)-\varphi(d)$$ that is 
 $$\liminf_{j\to\infty}  u_{{p_j}}(x_{p_j})=  \liminf_{j\to\infty}  \varphi(x_{p_j})> \varphi(d)$$
as $u_p(x_p)=\varphi(x_p), $    
$u_p(y_p)=\varphi(y_p)$  and $x_p\to d,$    $y_p\to e$
 by \eqref{step1}.

Now we prove that $$\limsup_{p\to\infty}\beta_p\leq\beta^*.$$
 Again by contradiction we suppose that there exists  a sequence $p_j\to \infty$ such that  $\limsup_{p\to\infty}\beta_p=r>\beta^*.$

Let $H=\{x\in (d,e), F \geq r\}$ and $L=\{x\in (d,e), F <r\}.$
Then we have $$\mu(L)>\frac{e-d-\varphi(d) +\varphi(e)}2$$ (see property 1) and 
$$\mu(H)=e-d -\mu(L) <e-d  -\frac{e-d-\varphi(d) +\varphi(e)}2=\frac{e-d+\varphi(d)-\varphi(e)}2.$$
By property 3, we obtain $$\liminf_{j\to\infty} \int_H\psi_{p_j}(\beta_{p_j} -F_{p_j}^{**}(x)) dx \geq - \mu (H) >   - \frac{e-d+\varphi(d)-\varphi(e)}2.$$
By property 4, we obtain $$\lim_{j\to\infty} \int_L \psi_{p_j}(\beta_{p_j} -F_{p_j}^{**}(x)) dx= \mu (L)>\frac{e-d-\varphi(d)+\varphi(e)}2.$$
As $$u_{p_j}(y_{p_j})-u_{p_j}(x_{p_j})= \int_{x_{p_j}}^{y_{p_j}} \psi_{p_j}(\beta_{p_j} -F_{p_j}^{**}(x)) dx,$$
 passing to the limit for $j\to\infty,$ we obtain 
$$  \liminf_{j\to\infty} ( u_{p_j}(y_{p_j})-u_{p_j}(x_{p_j}))>  \varphi(e)- \varphi(d)$$
$$  \liminf_{j\to\infty} u_{p_j}(y_{p_j})  -  \varphi(d)\geq  \liminf_{j\to\infty}  ( u_{p_j}(y_{p_j})-u_{p_j}(x_{p_j})) >  \varphi(e)- \varphi(d)$$
$$  \liminf_{j\to\infty} u_{p_j}(y_{p_j})  -  \varphi(e)>0$$  and this fact is a contradiction.

Now we prove that $|k|\leq 1$ where $k$ is defined in \eqref{cappa1}.
Let $[d, e]\subset \Omega.$   By property 1  ($ e-d=\mu(O_0)+\mu(O_-)+\mu(O_+)$)
$$\mu(O_-)=h(\beta^*) \leq \frac{e-d-\varphi(d)+\varphi(e)}2\leq \lim_{t\to (\beta^*)^+}h(t) =\mu(O_0)+\mu(O_-):$$
then 
$$0 \geq 2\mu(O_-)-  (e-d-\varphi(d)+\varphi(e))=  \mu(O_-)-  \mu(O_0)- \mu(O_+) -  (-\varphi(d)+\varphi(e)),$$
that is, 
$$   -\mu(O_0) \leq   -\mu(O_-)+ \mu(O_+) +  (-\varphi(d)+\varphi(e))$$
and $$0\leq 2\mu(O_0)+2\mu(O_-)-(e-d-\varphi(d)+\varphi(e)) =\mu(O_0)+\mu(O_-)- \mu(O_+) - (-\varphi(d)+\varphi(e)), $$ that is, 
$$\mu(O_0) \geq  -\mu(O_-)+\mu(O_+) +(-\varphi(d)+\varphi(e)).$$
Then, if $\mu(O_0)>0$
\begin {equation} \label{kos} 
 -1\leq k=\frac{ \mu(O_+)  -\mu(O_-)  -\varphi(d)+\varphi(e) } {\mu(O_0)} \leq 1.\end{equation}
Now we prove  that if  $\mu(O_0)>0$ then \begin{equation}\label{formulak}\lim_{p\to\infty} \psi_p(\beta_p-\beta^*)=k.\end{equation}
In fact, we have 
$$u_{p}(y_{p})-u_{p}(x_{p})= \int_{x_{p}}^{y_{p}}\psi_p(\beta_{p} -F_p^{**}(x)) dx=  \int_{x_{p}}^{d} \psi_p(\beta_{p} -F_p^{**}(x)) dx + \int_{e}^{y_{p}} \psi_p(\beta_{p} -F_p^{**}(x)) dx+$$
$$
 \int_{O_-} \psi_p(\beta_p -F_p^{**}(t))dt +  \int_{O_+} \psi_p(\beta_p -F_p^{**}(t))dt  + \psi_p(\beta_p -\beta^*-\int_d^{x_p} f dt  ) \mu(O_0).$$

As  $u_p(x_p)=\varphi(x_p), $    
$u_p(y_p)=\varphi(y_p)$  and $x_p\to d,$    $y_p\to e$
 by \eqref{step1}, by property 4, we obtain 
 $$    \lim_{p\to\infty} ( \varphi(y_p)- \varphi(x_p)) =\mu(O_-)-\mu(O_+) + \lim_{p\to\infty} \psi_p(\beta_p -\beta^* -\int_d^{x_p} f dt )\mu(O_0)$$
 and \eqref{formulak} is proved.
 
 For any   $x\in (d,e)$ we have, by \eqref{step1}, that 
 $x\in (x_p,y_p)$ (for  $p\geq p_0),$  and, by \eqref{eqHL2}, 
  $$ u_p(x)= \varphi(x_p)   +\int_{x_p} ^x \psi_p(\beta_p -F_p^{**}(t))dt = \int_{x_{p}}^{d} \psi_p(\beta_{p} -F_p^{**}(x)) dx +$$
$$\int^{x}_{d}   \chi_{O_-} \psi_p(\beta_p -F_p^{**}(t))dt +  \int^{x}_{d}   \chi_{O_+} \psi_p(\beta_p -F_p^{**}(t))dt  + \psi_p(\beta_p -\beta^*-\int_d^{x_p} f dt  ) \mu(O_0).$$

 By property 4 and  \eqref{formulak},
 passing to the limit,
\begin{equation}  \label{limj}    \lim_{p\to\infty}  u_p(x)=  \varphi(d)-  \int_d^x   \chi_{O_+}   dt + \int_d^x   \chi_{O_-}   dt + k \int_d^x   \chi_{O_0}   dt \end{equation}
 and we obtain \eqref{U}.
 
 To complete the proof of the theorem we have to consider the case $\varphi (d)-\varphi (e)=d-e.$  If $\varphi (d)-\varphi (e)=d-e$  then   \begin {equation*}
\beta^*= \sup\{ r\in \mathbb{R}: h(r)\leq d-e\}=+\infty \end {equation*} and then  $O_+=O_0=\emptyset$ and $O_-=(d,e).$   

By proceeding as in the  proof of  (\ref{contra2})  we show that  

\begin {equation} \label{contrabb} r= \liminf_{p\to\infty}\beta_p\geq F_{+,d}. \end{equation}

 Let the  sequence $p_j\to \infty$ be such that $\lim_{j\to\infty}\, \beta_{p_j}=r\geq F_{+,d}$  and denote by $O_-(r)=\{x\in (d,e), F <r \},\,\,O_0(r)=\{x\in (d,e), F =r \}$
and $O_+(r)=\{x\in (d,e), F >r \} $  then $O_+(r)= \emptyset.$ We discuss first the case  $r=F_{+,d}$

We  proceed as in the proof of (\ref{kos})  to  show that  if  $\mu(O_0(F_{+,d}))>0$ then $k=1$ where 

\begin{equation}\label{kappah11} \,\,\,
k= \begin {cases}
   \frac{   - \mu(O_-(F_{+,d}))  - \varphi(d)+ \varphi(e) } { \mu(O_0(F_{+,d}))}   \qquad &if \qquad   \mu(O_0(F_{+,d}))>0, \\
 0 \qquad &if \qquad   \mu(O_0(F_{+,d}))=0. \end{cases}
  \end{equation}
Analogously  we proceed as  in the proof of (\ref{formulak}) and of  (\ref{limj}) to show that if  $\mu(O_0(F_{+,d}))>0$ then \begin{equation}\label{formulakcc}\lim_{j\to\infty} \psi_{p_j}(\beta_{p_j}-F_{+,d})=1\end{equation}

\noindent and
 $$\lim_{j\to\infty}  u_{p_j}(x)=  \varphi(d) + \int_d^x   \chi_{O_-(F_{+,d})}   dt + \int_d^x   \chi_{O_0(F_{+,d})}   dt =\varphi(d)+x -d $$ \noindent and  \eqref{U}  is proved.

Finally  for any sequence $p_j\to \infty$ be such that $\lim_{j\to\infty}\, \beta_{p_j}=r^* > F_{+,d}$   we have $O_+(r^*)=O_0(r^*)= \emptyset,\,\,\, O_-(r^*)=(d,e), \,\,\,k=0$ and

 $$\lim_{j\to\infty}  u_{p_j}(x)= \varphi(d)+x -d $$ 

\noindent and  \eqref{U}  is proved.

\textbf{Step 2.} We remove assumption (\ref{up4}). We start by noticing that, if (\ref{up4}) does not hold, by property (\ref{limsup}) we deduce that 
$$ \limsup\frac{|u_p(y_p)-u_p(x_p)|}{y_p-x_p}=1 $$ as $p\to \infty.$    Actually  there exists the limit  (see \ref{step1}) $$\lim\frac{|u_p(y_p)-u_p(x_p)|}{y_p-x_p}=\lim\frac{|\varphi(y_p)-\varphi(x_p)|}{y_p-x_p}=\frac{|\varphi(e)-\varphi(d)|}{e-d}=1.$$
According to Remark 2.2, Theorem 2.1 and Proposition 2.1, the limit function  $U(x)$  is the affine function connecting the points $(d,\varphi(d))$ and $(e,\varphi(e))$ (see formula (\ref{U0})) which  coincides  with the function defined in (\ref{U}).

\textbf{Step 3.}  We discuss  assumption (\ref{step1}).
 As  the interval  $(d,e)$ is a (connected) component  of $ \Omega \setminus \Gamma_\infty$  (and  $[d,e]\subset (a,b)$) by the definition of $\Gamma_\infty$ there exist $x_p \in (a,b)$ and  $y^*_p\in (a, b)$ such that $x_p<y^*_p, $ \,\,\, $u_p(x_p)=\varphi(x_p)$ and $x_p\to d,$
 $u_p(y^*_p)=\varphi (y^*_p)$ and $y^*_p\to e.$
   We discuss now the property 
 \begin{equation}\label{equa}    -\triangle_p u_p=f  \,\, \text{in} \,\,(x_p, y_p) .\end{equation}
Let $z_p$ the first point  $ z_p\in (x_p,y^*_p)$  such that $u_p$ meets the obstacle i.e. $u_p(z_p)=\varphi(z_p).$
First we note that   $\limsup z_p\leq \lim y^*_p=e$ and   $\liminf z_p\geq \lim x_p=d$  hence if $\liminf z_p=e$ then  $z_p \to e,$  property (\ref{equa}) holds in the interval $(x_p, z_p)$ and we  choose $y_p=z_p$.

Furthermore if there exists a sequence $z_{p_j}$ converging to  some $\eta\in (d,e)$  such that $u_{p_j} (x)=\varphi (x),\,\, \forall  x\in [z_{p_j}, z_{p_j}+ \delta_{p_j}],\,\,\delta_{p_j}>0$ 
 then by assumption (\ref{Lambda})  we deduce that $\limsup\delta_{p_j} =0.$  In fact if  $\limsup \delta_{p_j} =\delta_0>0$  then there exists $ \delta >0 $ such that the interval $[\eta,\eta+  \delta]$ is contained in  $ \overline {int \Gamma_\infty^* }\bigcap (d,e)$  and this is a contradiction  with the fact that $(d,e)\bigcap\Gamma_\infty=\emptyset.$  
  If $\limsup \delta_{p_j} =0$ then the interval $[z_{p_j}, z_{p_j}+\delta_{p_j}]$ vanishes and the limit function $U(x)$ is not affected by these vanishing contacts (see Remark 2.2). 

If $\eta=d$  the interval $[x_{p_j}, z_{p_j}]$ vanishes and the limit function $U(x)$ is not affected by these vanishing contacts (see Remark 2.2). Similar arguments hold for the choice of the points  $x_p.$

\textbf{Step  4.}  If the interval  $(a,c)$ is a  (connected) component  of $\Omega\setminus \Gamma_\infty$  we proceed in a similar manner using  Proposition \ref{prop1}. If the interval  $(d,b)$ is a  (connected) component  of $\Omega\setminus \Gamma_\infty$  we proceed in a similar manner using  Proposition \ref{prop3}.
\end{proof}

\begin {remark}\label {nono} We note that  a theorem  analogous of Theorem \ref{main} holds for obstacle problems with non homogeneous boundary conditions. We skip the proof, which can be easily done  by modifying the proof of Theorem \ref{main} and taking into account the results of Section 2 concerning the Dirichlet problem with non homogeneous boundary conditions.
\end{remark}

\begin {remark}\label {nolim} We note a peculiarity of the limit of solutions to obstacle Problems (\ref{OPp}). If the right hand  term in the Lewy-Stampacchia inequality  \eqref{LS} is uniformly bounded, then (up to pass to a subsequence) there exists the weak  limit   $f^*$ of the functions  $-\bigtriangleup_pu_p.$ However  the limit  $U^*$ of the solutions  $u_p^*$ of  Dirichlet Problems  (\ref{DP})  with datum  $f^*$  may  not  coincide  with the limit  of the solutions to obstacle Problems (\ref{OPp}).  We can construct examples in which   $U^*$  belongs to the convex $\mathcal{K}^\infty$ but it is not  a maximizer of  (\ref{eq:3DV}) (Example 2 in Section \ref{example} ) as well as examples in which   $U^*$ does not  belongs to the convex (Example 1 in Section \ref{example}).
\end{remark}

\section{$n$-dimensional  Obstacle problem}

 First we consider the radial case.
 
 Let $\Omega$ be the annulus $B_{r_1,r_2}:=\{x\in \mathbb{R}^n,  r_1<|x|< r_2      \},$  $ 0<r_1<r_2,$   \begin {equation} \label {iporad} f(x)=g(|x|)  \quad \text{and}  \quad\varphi(x)=\Phi(|x|). \end{equation}

\begin{tm} \label {radial} Suppose that  (\ref{f}), (\ref{triangle}), (\ref{knonv}),  (\ref{Lambda}) and \eqref{iporad} hold.
Then the solutions $u_{p}$  of Problems (\ref{OPp}) converge uniformly to the following function $U\in \mathcal{K}^{\infty}:$
 $$ U(x)=\varphi(x)\, \text{in}\,\, \Gamma_\infty $$ and  for any (connected) component $B_{d,e}$ of $ \Omega \setminus \Gamma_\infty$ such that $[d, e]\subset(r_1,r_2)$   
 \begin {equation} \label {Ur}  U(x)=\int_d^{|x|} (\chi_{O_-}-\chi_{O_+}+k \chi_{O_0}) dt +\Phi(d)\end {equation} 
where 
$$O_-=\{t\in (d,e), G <\beta^* \}, \quad\quad O_+=\{t\in (d,e), G >\beta^* \},\quad\quad O_0=\{t\in (d,e), G =\beta^* \}$$
$$G(t)=\int_d^t  \tau^{n-1} g(\tau)d\tau, \quad \quad h(r)=\mu(\{t\in (d,e) : G(t)<r\})$$

\begin {equation} \label {star1r}
\beta^*= \sup\{ r\in \mathbb{R}: h(r)\leq\frac{e-d-\Phi(d)+\Phi(e)}2\}\, 
\end {equation}

\begin {equation} \label {cappa1r}
k= \begin{cases}
   \frac{ \mu(O_+)  -\mu(O_-)  -\Phi(d)+\Phi(e) } {\mu(O_0)}   \qquad if \qquad  \mu(O_0)>0,\\
 0 \qquad if \qquad  \mu(O_0)=0.\\
\end{cases}
\end {equation}

For any (connected) component $B_{r_1,c}$ of $ \Omega \setminus \Gamma_\infty$    
 \begin {equation} \label {U1r}  U(x)=\int_{r_1}^{|x|} (\chi_{O_-}-\chi_{O_+}+k \chi_{O_0}) dt \end {equation} 
where 
$$O_-=\{t\in (r_1,c), G <\beta^* \}, \quad\quad O_+=\{t\in (r_1,c), G >\beta^* \},\quad\quad O_0=\{t\in (r_1,c), G =\beta^* \}$$
$$G(t)=\int_{r_1}^t \tau^{n-1} g(\tau)d\tau, \quad \quad h(r)=\mu(\{t\in (r_1,c) : G(t)<r\})$$

\begin {equation} \label {star2r}
\beta^*= \sup\{ r\in \mathbb{R}: h(r)\leq\frac{c- r_1+\Phi(c)}2\}\,
\end {equation}

\begin {equation} \label {cappa2r}
k= \begin{cases}
  \frac{ \mu(O_+)  -\mu(O_-)  +\Phi(c) } {\mu(O_0)}   \qquad if \qquad  \mu(O_0)>0,  \\  
 0 \qquad if \qquad  \mu(O_0)=0.\\
\end{cases}
\end {equation}

For any (connected) component    $B_{d,r_2}$   of $ \Omega \setminus \Gamma_\infty$  
 \begin {equation} \label {U2r}  U(x)=\int_{r_2}^{|x|} (\chi_{O_-}-\chi_{O_+}+k \chi_{O_0}) dt \end {equation} 
where 
$$O_-=\{t\in (d,r_2), G <\beta^* \}, \quad\quad O_+=\{t\in (d,r_2), G>\beta^* \},\quad\quad O_0=\{t\in (d,r_2), F =\beta^* \}$$
$$G(t)=\int_{r_2}^t  \tau^{n-1} g(\tau)d\tau,  \quad \quad h(r)=\mu(\{t\in (d,r_2) : G(t)<r\})$$

\begin {equation} \label {star3r}
\beta^*= \sup\{ r\in \mathbb{R}: h(r)\leq\frac{r_2-d-\Phi(d)}2\}\, 
\end {equation}

\begin {equation} \label {cappa3r}
k= \begin{cases}
   \frac{ \mu(O_+)  -\mu(O_-)  -\Phi(d) } {\mu(O_0)}   \qquad if \qquad  \mu(O_0)>0,\\
 0 \qquad if \qquad  \mu(O_0)=0.\\
\end{cases}
\end {equation}

  \end{tm}

We skip the proof, as it is very similar to the proof of Theorem \ref{main}.
We note that in the previous results the solutions $u_p$ converge uniformly to the function $U$ as $p\to\infty$  even if  Problem $\eqref{eq:3DV}$ does not have unique solution.

\begin{remark}\label{cerchio} If $\Omega$ is the ball $B_{r}:=\{x\in \mathbb{R}^n, |x|< r     \},$  $r>0,$ then under the assumptions of Theorem \ref{radial}, 
 the same results hold  except for the case of the (connected) component     $B_{0,c}$  of $ \Omega \setminus \Gamma_\infty$    where formula \eqref{U1r} becomes
 
 \begin {equation} \label {U1r0}  U(x)=\int_{c}^{|x|} (\chi_{O_-}-\chi_{O_+}) dt + \Phi(c)\end {equation} 
where 
$$O_-=\{t\in (0,c), G <0 \}, \quad\quad O_+=\{t\in (0,c), G >0 \},$$
$$G(t)=\int_{0}^t \tau^{n-1} g(\tau)d\tau.$$
\end{remark}
The following results concern arbitrary domains  hence  as we do not assume any smoothness condition on the boundaries  these results hold true for \emph{bad domains} as the Koch Islands (see \cite {CF}, \cite {CV} and \cite {CFV}).

We denote  $$\mathcal{M}_\varphi=\{u\in \mathcal{K}^\infty: \mathcal{F}(u)=\max_{w\in \mathcal{K}^\infty} \mathcal{F}(w)\}$$
and 
\begin {equation} \label {AMS} 
\mathcal{A}_\varphi=\{  u\in C(\bar{\Omega}): \text  {there exists a sequence} \,\, p_j\to \infty \quad \text{such that}\,\,  u_{p_j} \to u \quad \text{in} \quad C(\bar{\Omega}) \}\end {equation}  
where $u_p$ denotes the  solution to $(\ref{OPp})$.

Condition (\ref{Lambda}) is satisfied in all the examples of Section 
\ref{example} and Theorem 2.8 in \cite{RTU} as well as Theorems  7.5  and
1.3 in \cite{BDR} provide sufficient conditions to assure that condition
(\ref{Lambda}) holds true. However in \cite{RTU} and \cite{BDR} 
smoothness assumptions are required while in our paper we deal with a
larger class of obstacles and data. Then we are interested in proving that
the set $\mathcal{A}_\varphi $  defined in (\ref{AMS}) is a singleton by 
a different approach
(see Theorems  \ref{tf},  \ref{tf3},   \ref{tf2} and  \ref{tf4}).  
Condition (\ref{f3})  in Theorem  \ref{tf},  that  concerns data $f$ 
changing sign in $\Omega,$ puts in relation  the position of the support 
of $f$  with respect   to the  boundary of  $\Omega$ and provides an alternative
assumption that, in some sense, forces the coincidence sets  to have a
\emph{good} behavior.  Similarly, the sign conditions  on the datum $f$  in
Theorems    \ref{tf3},   \ref{tf2} provide  alternative assumptions.
Finally, we recall that, as the constraint  in the convex $\mathcal{K}$ is
from below, then  as a consequence of the Lewy-Stampacchia inequality
\eqref{LS}, the  easy situation is when $f $  (possibly vanishing  in a
set of positive measure) is non negative  while, when $f $ is non
positive, we have to require also  conditions on $-\triangle_p\varphi$ 
(see \eqref{f115}  and  \eqref{f116} respectively).

\begin{tm}
\label{tf} Suppose that  \eqref{f},  \eqref{triangle}  and \eqref{knonv}  hold,  and 
\begin {equation} \label {f2}\Omega_+ \quad \text{and}\quad  \Omega_-  \,  \text{are open connected and non empty} \end{equation}
where $$\Omega_+=\{  x\in \Omega, f(x)>0 \} \quad \text{and}\quad  \Omega_-=\{  x\in \Omega, f(x)<0 \} $$
and
\begin {equation} \label {f3}     \inf_{x\in \Omega_+ } \sup_{y\in \Omega_- }(d(x)+d(y)-|x-y|) \leq 0  \end{equation}
where $d(x)$ denotes the distance of $x$ from the boundary.  Then the set $\mathcal{A}_\varphi $  defined in (\ref{AMS}) is a singleton.
\end{tm}

We just observed $\mathcal{A}_{\varphi}\subset\mathcal{ M}_{\varphi}$ and before proving this theorem, we state same preliminary results.

 \begin{prop}
 Let $u\in \mathcal{M}_\varphi$ then 
 \begin{equation} \label {u1} u(x)=\inf\{   u(y)+|x-y|, y\in \Omega_- \cup \partial \Omega\}, \forall x\in \Omega\end{equation}

  \begin{equation} \label {u2} u(x)=\sup\{   u(y)-|x-y|, y\in \Omega_+ \cup \partial \Omega\}  \vee\varphi^*(x), \forall x\in \Omega\end{equation}
  
  where $$\varphi^*(x)=\sup\{   \varphi (y)-|x-y|, y\in \Omega\} .$$
 
  \end{prop}

  \begin{proof}
We prove \eqref{u2}, as \eqref{u1} is similar (see Proposition 6.1 in  \cite{HL}).
Let  $$w(x)= \sup\{   u(y)-|x-y|, y\in \Omega_+ \cup \partial \Omega\}\vee\varphi^*(x).$$
As $u\in Lip_1(\bar\Omega)$ and $u\geq \varphi$
we deduce $u\geq w.$  Moreover $w\in Lip_1(\bar\Omega),$  $u=w$ on $\Omega_+\cup \partial \Omega$ and $w\in \mathcal{K}^\infty.$ 
Then as $$\int _{\Omega_+}f w dx +\int _{\Omega_-}f w dx =\mathcal{ F}(w)\leq \mathcal{F}(u)=\int _{\Omega_+}f u dx +\int _{\Omega_-}f u dx $$
we obtain $$\int _{\Omega_-}f (u-w) dx\geq 0$$ and so $u=w$ on $\Omega_-.$

  \end{proof}

  By proceeding as in the proof Propositions 6.4, 6.5, 6.6 and 6.7 of   \cite{HL}
   we obtain la following result.

 \begin{prop}
 For any $u,v\in \mathcal{M}_\varphi$ we  have
  \begin{equation} \label {pm} \sup_{\Omega_+} (u-v)^+  = \sup_{\Omega_-} (u-v)^+  \end{equation}
  and  \begin{equation} \label {nabla} \nabla u=\nabla v \qquad \text{a. e. in   } \Omega_+\end{equation}
\end{prop}

 Now we prove Theorem \ref{tf}.

  \begin{proof}
  First we show, that for any functions $u, v\in \mathcal{M}_\varphi$   \begin{equation} \label {sp} u=v\, \text{on  }supp f. \end{equation}

  By contradiction we suppose  $\sup_{\Omega_+} (u-v)^+  =h>0$ then by \eqref{nabla} we obtain that $u(x)=v(x)+h$ for any $x\in  \Omega_+.$
  
  By \eqref{f3}, we deduce that for any $\varepsilon >0$, there exists a point $x_\varepsilon$ in $\Omega_+$ such that 
  $$d(x_\varepsilon )+d(y)-|x_\varepsilon -y| \leq \varepsilon$$ for any $y\in \Omega_-.$ 
  By using that $u, v\in Lip_1(\bar\Omega)$  vanish on the boundary $\partial \Omega $ and property \eqref{u1}, we deduce 
  $$ u(x_\varepsilon)\leq d(x_\varepsilon) \leq \varepsilon +v(x_\varepsilon)$$ and this is a contradiction if $\varepsilon\in (0, h).$
  Then $u(x)=v(x)$ for any $x\in  \Omega_+.$ 
  By  \eqref{pm} we deduce that $u(x)\leq v(x)$ for any $x\in  \Omega_-.$  By changing the role of  $u$ and  $v$ in \eqref{pm} 
  we obtain $v(x)\leq u(x)$ for any $x\in  \Omega_-$  and this completes the proof of \eqref{sp}.
  
  Now,  according to \cite{RTU},  for any $u\in \mathcal{A}_\varphi$  we denote by  $\Gamma_u=\{x\in \Omega\setminus supp f: u(x)=\varphi(x)\}$
 then 
 \begin{equation} \label {pmu} \begin {cases}
 u(x)\geq \varphi(x), &\quad \text{in} \quad\Omega\setminus supp f\\
 -\triangle_\infty u=0 &\quad \text{in} \quad\Omega\setminus (supp f\bigcup \Gamma_u) \quad\text{in the viscosity sense}\\
 -\triangle_\infty u\geq0 &\quad\text{in}\quad \Omega\setminus supp f \quad\text{in the viscosity sense}. 
  \end {cases}\end {equation}

Now we denote by 
$$w(x)= \inf \{ v(x): v\in \mathcal{G}\}$$ where  $\mathcal{G}$ denotes the set of the continuous functions that are infinity super-harmonic in $ \Omega\setminus supp f$ and satisfy the conditions
   $v(x)\geq \varphi(x), \quad \text{in} \quad\Omega\setminus supp f $ and $v=u$ on  $\partial (\Omega\setminus supp f ).$
We note that $u\in \mathcal{G}$  and $w$ is upper semicontinuous  and infinity super-harmonic in $ \Omega\setminus supp f.$  Moreover  $u\geq w.$
  
  We consider the open set $$W=\{x\in  \Omega\setminus supp f: u(x)>w(x)\}.$$We have  $u(x)=w(x)$ on $ \partial W$ and $u(x)>w(x)\geq \varphi$ in $W$ so  $W\subset   \Omega\setminus (supp f\bigcup \Gamma_u)$ then $u$ is infinity harmonic in $W.$ By the comparison principle  (see for instance  \cite{J}) we conclude that $u\leq w$ in  $W.$ 
 Hence $W=\emptyset$ and $u=w$ in   $\Omega\setminus supp f .$
  
  Moreover  any element $v\in \mathcal{A}_\varphi$  belongs to   $\mathcal{G}$  as $u=v=0$ on $\partial \Omega$ and by  (\ref{sp}) we have  $u=v$ on $supp f,$  hence $ u\leq v.$  By the same argument we can show that   $ v\leq u $  then $u= v$ on $ \Omega\setminus supp f.$ This completes the proof.
 \end{proof}

We now discuss the situation  in which  the datum $f$ does not change sign in $\Omega.$ 
We note that  $  f(x)\geq \delta _0>0$ then $\mathcal{A}_\varphi=\{d(x) \} $  and  in particular  the set $\mathcal{A}_\varphi $ is a singleton.  In fact we consider the  Dirichlet problem:
\begin{equation} \label{Dpn}
\text{find } u\in W_0^{1,p}(\Omega),\, \,\,\,\int_{\Omega} |\nabla u|^{p-2}\nabla u\nabla v\,dx\, -\int_{\Omega} fv\,dx\,= 0\,\,\,
  \forall v\in W_0^{1,p}(\Omega).
  \end{equation}
 If we assume that $ f\in L^\infty(\Omega) $ then, for any fixed $p,$ there exist an unique solution $u_{p,D}$ of Problem (\ref{Dpn}).
 We denote  
 $$\mathcal{M}=\{u\in  W_0^{1,\infty}(\Omega)\cap Lip_1(\bar{\Omega}): \mathcal{F}(u)=\max_{w\in  W_0^{1,\infty}(\Omega)\cap Lip_1(\bar{\Omega})} \mathcal{F}(w)\}$$
and $$\mathcal{A}=\{  u\in C(\bar{\Omega}): \text  {there exists a sequence} \,\, p_j\to \infty \quad \text{such that}\,\,  u_{p_j,D} \to u \quad \text{in} \quad C(\bar{\Omega}) \}$$
 where $u_{p,D}$ denotes the  solution to (\ref{Dpn}). 
   If $  f(x)\geq \delta _0>0$ then  there exists the limit of the functions $u_{p,D}$ in $C(\bar{\Omega})$  and we have $\lim_{p\to\infty}u_{p,D}(x)= d(x)$  where $d(x)$ denotes the distance of $x$ from the boundary (see  Proposition 5.2 in \cite{BDM} and \cite{HL}). 
   
  By Lewy-Stampacchia inequality  (\ref{LS})    the solutions $u_p$ of Problem (\ref{OPp})  solve Dirichlet problems with data $f_p\geq \delta _0>0.$ Then  by arguing as in the proof of 
  Proposition 5.2 in \cite{BDM} we deduce that 
     $\mathcal{A}_\varphi=\mathcal{A}=\{d(x) \}$  and, in particular, the set $\mathcal{A}_\varphi $ is a singleton.

The following theorem concerns the case $f(x)\geq 0.$
\begin{tm}
\label{tf3}Suppose that  \eqref{f},  \eqref{triangle}  and \eqref{knonv}  hold, and 
\begin {equation} \label {f113} f\geq 0. \end{equation}
 Then the set $\mathcal{A}_\varphi $ is a singleton.
\end{tm}
\begin{proof}
 For any functions $u, v\in \mathcal{A}_\varphi,$  using the Lewy-Stampacchia inequality  (\ref{LS})  and repeating the previous argument we  show that   
   \begin{equation} \label {sp3} u=v= d(x) \, \text{on} \,\,supp f   \end{equation} 
  (see  Proposition 5.2 in \cite{BDM} and \cite{HL}).

  Now we proceed as in the proof of Theorem \ref{tf}  (see also \cite{RTU})  to  conclude the proof.
   \end{proof}

 By the same arguments we deal with  $f$ negative more precisely the following result holds true.

\begin{tm}
\label{tf2} Suppose that assumptions \eqref{f},  \eqref{triangle} and \eqref{knonv}  hold, and 
\begin {equation} \label {f11}   f(x)\leq -\delta _0<0, \end{equation}
  then the set $\mathcal{A}_\varphi $ is a singleton.
\end{tm}

The situation in which   the datum $f$ is not positive, is more delicate. In the following theorems different conditions on the obstacle are assumed. In Section 5 we see an example of obstacle problem where  the assumptions of  Theorem \ref{tf5} are satisfied (Example 2).

\begin{tm}
\label{tf4}Suppose that assumptions \eqref{f},  \eqref{triangle} and \eqref{knonv}  hold, and 
\begin {equation} \label {f114} f\leq 0,  \end{equation} 
and
\begin {equation} \label {f115}\ \, -\triangle_p\varphi \leq C_0<0,\,\,\forall p,
  \end{equation}
 then the set $\mathcal{A}_\varphi $ is a singleton.
\end{tm}
\begin{proof}
 
  For any functions $u, v\in \mathcal{A}_\varphi$ , \,by  (\ref{LS}) and   (\ref{f115}) we have   
   \begin{equation} \label {sp4} u=v= -d(x) \, \text{in } supp f   \end{equation} 
  (see  Proposition 5.2 in \cite{BDM} and \cite{HL}). 
   Now we proceed as in the proof of Theorem \ref{tf} to  conclude the proof.
   \end{proof}

  \begin{tm}
\label{tf5}Suppose that assumptions \eqref{f},  \eqref{triangle}, \eqref{knonv},  \eqref{Lambda} and  \eqref{f114} hold.  
 Furthermore  we assume that the set $ \Omega_-=\{  x\in \Omega, f(x)<0 \}$ is open  and 
\begin {equation} \label {f116}\ \, -\triangle_p\varphi \geq 0,\,\,\forall p,
  \end{equation}
 then the set $\mathcal{A}_\varphi $ is a singleton.
\end{tm}
\begin{proof}
 
  For any functions $u, v\in \mathcal{A}_\varphi,$   we have   
  \begin{equation} \label {sp1} u=v= -d(x) \, \text{in }  supp f \setminus int(\Gamma_\infty^*)  \end{equation} 

 In fact  for any  $B(\hat{x}, \delta) \subset \Omega_-\setminus \Gamma_\infty^*$   we have $B(\hat{x}, \delta)\bigcap \Gamma_\infty^*=\emptyset$ and then $B(\hat{x}, \delta)\bigcap \Gamma_p=\emptyset$ (for large $p$) and we can  use Proposition 5.2 in \cite{BDM}. 
 
 We set $\Omega^*=\Omega\setminus \big(supp f \setminus int(\Gamma_\infty^*)\big),$ according to \cite{RTU},  for any $u\in \mathcal{A}_\varphi$  we denote by  $\Gamma_u=\{x\in \Omega^*: u(x)=\varphi(x)\}$ and we have
 \begin{equation} \label {pmu5} \begin {cases}
 u(x)\geq \varphi(x), \quad \text{in} \quad\Omega^*\\
 -\triangle_\infty u=0\quad \text{in} \quad\Omega^*\setminus \Gamma_u \quad\text{in the viscosity sense}\\
 -\triangle_\infty u\geq0 \quad\text{in}\quad \Omega^*\quad\text{in the viscosity sense}. 
  \end {cases}\end {equation}

In fact, for $u\in \mathcal{A}_\varphi,$  and $\hat{x} \in \Omega^*\setminus \Gamma_u$  we have $u(\hat{x})>\varphi (\hat{x})$ then there exists a ball   $B(\hat{x}, \delta)$  such that $u(x)>\varphi (x)$ for any $x\in  B(\hat{x}, \delta)$ and hence $u_{p_k}(x)>\varphi (x)$  for any $x\in B(\hat{x}, \delta)$   (for $k$  large)  then  $B(\hat{x}, \delta)\bigcap \Gamma_{p_k}=\emptyset.$  As a consequence  $B(\hat{x}, \delta)\bigcap \Gamma_{\infty}=\emptyset$  and  (see \eqref{Lambda})  we deduce  $f=0$  in $\Omega^*\setminus \Gamma_u.$ 
    
    Moreover  for any ball  $B(\hat{x}, \delta)\subset \Omega^*\cap  \{x\in \Omega:  f(x)<0\}$ we have  $B(\hat{x}, \delta)\subset int(\Gamma_\infty^*).$  By  \eqref{Lambda} we deduce that $B(\hat{x}, \delta)\subset \Gamma_\infty $  and then  there exists $p_0$ such that  $u_p(x)=\varphi(x)$ for any $p\geq p_0$  and then by \eqref{f116}  $ -\triangle_p u_p= -\triangle_p \varphi \geq0.$   Now we denote by 
 $$w(x)= \inf \{ v(x): v\in \mathcal{G}\}$$ where  $\mathcal{G}$ denotes the set of the continuous functions that are infinity super-harmonic in $ \Omega^*$ and satisfy the conditions
   $v(x)\geq \varphi(x), \quad \text{in} \quad\Omega^* $ and $v=u$ on  $\partial (\Omega^*).$
 By proceeding as in the proof of Theorem \ref{tf}  we conclude the proof.  \end{proof}

   \section{Examples}\label{example}

In this section, we provide examples, comments and remarks.

\textbf{ Example 1}

 Let $f=\chi_{(1,\frac32)}-\chi_{(\frac32,2)},$ \,\, $\Omega=(0,3).$
The solution to  (\ref{DP}) with  homogeneous Dirichlet conditions, is

 $$u_{p,D}=\begin{cases} c^\beta x &\qquad  0\leq x\leq 1\\
  -\frac{(-x+c+1)}{\beta+1} ^{\beta+1}+ c^\beta+\frac{c^{\beta+1}}{\beta+1} &\qquad
1<x\leq  c+1 \\
   -\frac{(x-c-1)}{\beta+1} ^{\beta+1}+ c^\beta+\frac{c^{\beta+1}}{\beta+1}  &\qquad 
c+1<x\leq \frac32 \\
  \frac{(-x-c+2)}{\beta+1} ^{\beta+1}- c^\beta-\frac{c^{\beta+1}}{\beta+1} &\qquad
\frac32<x\leq 2-c \\
 \frac{(x+c-2)}{\beta+1} ^{\beta+1}- c^\beta-\frac{c^{\beta+1}}{\beta+1}  &\qquad
2-c<x\leq 2 \\
  c^\beta(x-3) &\qquad  2<x\leq 3
 \end{cases}$$
where $$\beta=\frac1{p-1} $$
and   \begin{equation}\label{cdir} c^\beta+\frac{c^{\beta+1}}{\beta+1} =   \frac{(\frac12 -c)}{\beta+1}
^{\beta+1}.  \end{equation}
When  $p\to\infty$,  from \eqref{cdir}  we obtain  $c\to 0,$ $ c^\beta\to \frac12$  and  $u_{p,D}$ tends to 
 \begin{equation}\label{DinftyE1} u_{\infty, D}=\begin{cases} \frac x2 &\qquad  0\leq x\leq 1\\ \frac32-x  &\qquad 1<x\leq
2\\\frac x2 -\frac32 &\qquad  2<x\leq3.
 \end{cases} \end{equation}
 We now consider the obstacle  $\varphi=0.$ The solutions to the variational  inequality  (\ref{OPp}) is 
  \begin{equation}\label{DV1E1}u_p=\begin{cases} c_p^\beta x &\qquad  0\leq x\leq 1\\
  -\frac{(-x+c_p+1)}{\beta+1} ^{\beta+1}+ c_p^\beta+\frac{c_p^{\beta+1}}{\beta+1}
&\qquad 1<x\leq  c_p+1 \\
   -\frac{(x-c_p-1)}{\beta+1} ^{\beta+1}+ c_p^\beta+\frac{c_p^{\beta+1}}{\beta+1} 
&\qquad  c_p+1<x\leq \frac32 \\
  \frac{(-x-c_p+2)}{\beta+1} ^{\beta+1} &\qquad \frac32<x\leq 2-c_p \\
0  &\qquad 2-c_p\leq x\leq 3 
 \end{cases}\end{equation}
where $$\beta=\frac1{p-1} $$
and \begin{equation}\label{cp} c_p^\beta+\frac{c_p^{\beta+1}}{\beta+1} =  2
\frac{(\frac12 -c_p)}{\beta+1} ^{\beta+1}.\end{equation} 

As  $p\to\infty$, from \eqref{cp}, we obtain that  $c_p\to 0,$  $ c_p^\beta\to 1$  and the  limit of functions $u_p$  is 

\begin{equation}\label{DVinftyE1}U=\begin{cases}  x &\qquad  0\leq x\leq 1\\
 2-x&\qquad 1<x\leq 2 \\
 0  &\qquad 2\leq x\leq 3 .
 \end{cases}\end{equation}

 In this example, all assumptions of Theorem \ref{main} are satisfied and  in particular $\Gamma_p=[2-c_p,3],$\,\,\,  $c_p\to0^+$  and  $\lim \Gamma_p=[2,3]=\Gamma_\infty.$

We note that condition  \eqref{f3} does not hold then this example  shows that  condition (\ref{Lambda}) can be satisfied even if  assumption (\ref{f3}) is not satisfied.

\begin {remark}  \label {E1} From this example we deduce that a  solution to Problem \eqref{eq:3DV} cannot be
obtained by taking the supremum between the obstacle and the \textit{variational solution} limit of the $u_{p, D}.$
In fact $$\mathcal{F}(u^+_{\infty, D})=\frac18<\mathcal{F}(U)=\frac14.$$
\end{remark}

\begin {remark}   \label {E11}  We observe that  in this example Problem (\ref{eq:3DNO}) does not have a unique
solution in $\mathcal{K}_D^{\infty} $  as $$\mathcal{F}(u_{\infty,D})=\mathcal{F}(U)=\frac14.$$ Theorem \ref{TD} selects the \textit{ variational solution}, limit of the $u_{p, D}.$
  In an analogous way,   problem \eqref{eq:3DV} does not have a unique solution in
$\mathcal{K}^{\infty} $  as $$\mathcal{F}(v)=\mathcal{F}(U)=\frac14$$
where  $$v=\begin{cases}  x &\qquad  0\leq x\leq 1\\
 2-x &\qquad 1<x\leq 2 \\
 x-2 &\qquad 2<x\leq 5/2 \\
  -x+3  &\qquad \frac52<x\leq 3.  
\end{cases}$$
Theorem  \ref{main}  selects the \textit{variational solution}, limit of the functions $u_{p}.$ 
 \end{remark}

\begin{figure}
\begin{center}
\includegraphics[width=4.5cm]{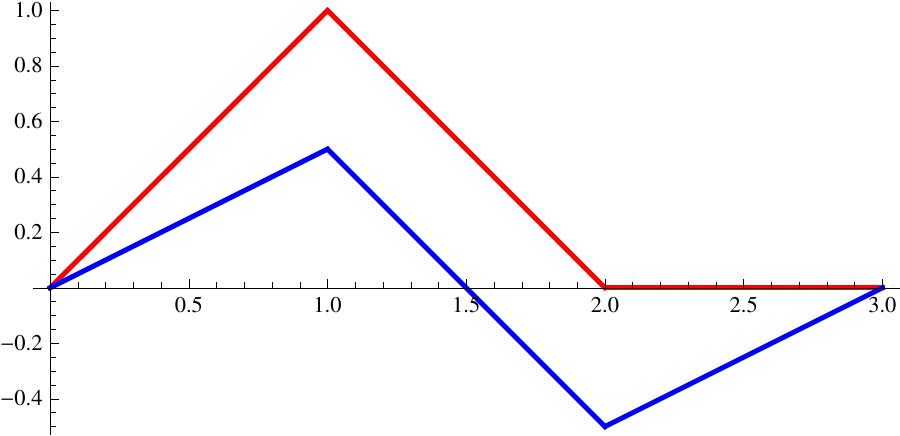}\
\end{center}
\caption{ Example 1: $U$ in red, $u_{\infty, D}$ in blue.  }
\label{Figure1}
\end{figure}

\textbf{Example 2}

 Let $f=-\chi_{(0,1)},$ $\Omega=(0,3).$
  Now we consider the obstacle  $\varphi=-\frac{1}{2}.$ The solution to the variational  inequality   \eqref{OPp} is 
 
 \begin{equation}\label{DV1E12} u_p=\begin{cases}  \frac{ ((\frac{\beta +1}{2})^{\frac1{\beta +1}}-x)^{\beta +1}- \frac{\beta +1}{2}}{\beta +1} &\qquad  0\leq x\leq ( \frac{\beta +1}{2})^{\frac{1}{\beta +1}} \\
 - \frac{1}{2} &\qquad  (\frac{\beta +1}{2})^{\frac{1}{\beta +1}}<x\leq c_p\\
 - \frac{1}{2}   + \frac{(x-c_p)^{\beta+1} }{\beta+1}    &\qquad c_p<x\leq 1\\
   (1-c_p)^  \beta (x-3)  &\qquad 1<x\leq 3 \end{cases}\end{equation}
where $$\beta=\frac1{p-1} $$
and \begin{equation}\label{cp2} \frac12 =  
2(1 -c_p)^\beta+ \frac{(1 -c_p)}{\beta+1} ^{\beta+1}.\end{equation} 

As  $p\to\infty$, from   \eqref{cp2}, we obtain that  $c_p\to 1^-,$ $ (1-c_p)^\beta\to \frac14$  and the  limit of functions $u_p$  is 
 \begin{equation}\label{cp22} U=\begin{cases}  -x  &\qquad  0\leq x\leq \frac12\\
 - \frac12 &\qquad \frac12<x\leq 1 \\
\frac14(x-3)  &\qquad 1< x\leq 3. 
 \end{cases}\end{equation} 

While the limit of the solutions to Problems (\ref{DP}) with  homogeneous Dirichlet conditions is

 \begin{equation}\label{DinftyE2} u_{\infty, D}=\begin{cases} -x \qquad  0\leq x\leq 1\\ 
 \frac{x-3}{2}  \qquad 1<x\leq 3.
\end{cases} \end{equation}

We note that, in this example, all assumptions of Theorems \ref{main}  and \ref{tf5}  are satisfied, in particular $\Gamma_p=[(\frac{\beta +1}{2})^{\frac{1}{\beta +1}},c_p],$\,\,\,  $c_p\to1^-$  and  $\lim \Gamma_p=[\frac12,1]=\Gamma_\infty.$
  
\begin {remark}  \label {E2} We note a peculiarity of the limit of solutions to obstacle problems (\ref{OPp}).

 In Example 1  the functions $u_p$ in (\ref{DV1E1}) converge to $U$ in (\ref{DVinftyE1})  while  the functions $-\bigtriangleup_pu_p$  converge to $f^*=\chi_{(1,\frac32)}-\chi_{(\frac32,2)}.$  Hence the limit  $U^*$ of the solutions  $u_p^*$ of the Dirichlet problem  (\ref{DP})  with datum  $f^*$ and homogeneous Dirichlet conditions,  coincides with the  function  $u_{\infty,D}$ in  (\ref{DinftyE1})  that  does not  belong to the convex $\mathcal{K}^\infty.$ 
 
 In Example 2  the functions $u_p$ in (\ref{DV1E12}) converge to $U$ in (\ref{cp22})  while  the functions $-\bigtriangleup_pu_p$  converge to $f^*=-\chi_{(0,\frac12)}.$  Hence the limit  $U^*$ of the solutions  $u_p^*$ of the Dirichlet problem  (\ref{DP})  with datum  $f^*$ and homogeneous Dirichlet conditions, is 
 
 $$U^*=\begin{cases}  -x  &\qquad  0<x\leq \frac12\\
 - \frac12 +\frac15(x-\frac12) &\qquad \frac12<x\leq 3 
 \end{cases}$$
that   belongs to the convex $\mathcal{K}^\infty,$  but it is not  a mazimizer of  (\ref{eq:3DV}).
\end{remark}

\begin{figure}
\begin{center}
\includegraphics[width=4.5cm]{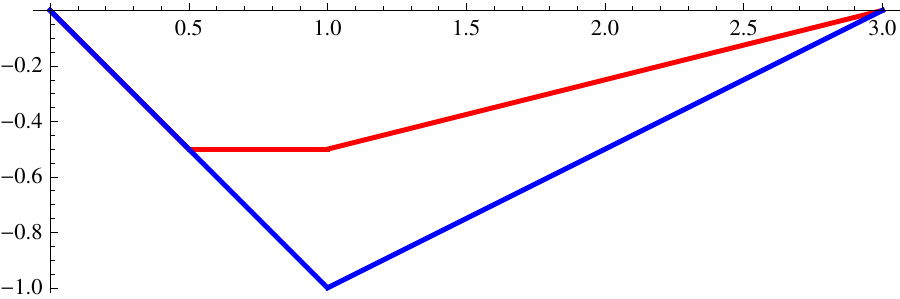}\
\end{center}
\caption{ Example 2: $U$ in red, $u_{\infty, D}$ in blue.   }
\label{Figure2}
\end{figure}

\textbf{Example 3} 

Let $f=\chi_{(0,1)}-\chi_{(2,3)},$ \,\, $\Omega=(0,3).$

The limit of the solutions to  (\ref{DP}) with  homogeneous Dirichlet conditions, is 

 \begin{equation}\label{DinftyE3} u_{\infty, D}=\begin{cases} x &\qquad  0\leq x\leq \frac34\\
  \frac64-x  &\qquad \frac34<x\leq\frac94\\
x-3 &\qquad  \frac94<x\leq3.
 \end{cases} \end{equation}

  Now we consider the obstacle  $\varphi=0.$ The solutions to the variational  inequality  (\ref{OPp}) is 
 
  \begin{equation}\label{DV1E3}u_p=\begin{cases} \frac{c_p^{\beta+1} - (c_p- x)^{\beta+1} }{{\beta+1} } &\qquad  0\leq x\leq c_p\\
  \frac{c_p^{\beta+1} - ( x -c_p)^{\beta+1} }{{\beta+1} }
&\qquad c_p<x\leq 1 \\
 (1-x)(1-c_p)^\beta + \frac{c_p^{\beta+1} - (1-c_p)^{\beta+1} }{{\beta+1} }
&\qquad  1<x\leq 2 \\
  \frac{(3-x-c_p)}{\beta+1} ^{\beta+1} &\qquad 2<x\leq 3-c_p \\
0 & \qquad 3-c_p< x\leq 3
 \end{cases}\end{equation}
where $$\beta=\frac1{p-1} $$
and \begin{equation}\label{cp3} -(1-c_p)^\beta+\frac{c_p^{\beta+1}}{\beta+1} =  2
\frac{(1 -c_p)}{\beta+1} ^{\beta+1}.\end{equation} 

As  as $p\to \infty,$
from \eqref{cp3}, we obtain that  $c_p\to 1,$ $ (1-c_p)^\beta\to 1$
and the  limit of functions $u_p$  is 

\begin{equation}\label{DVinftyE3}U=\begin{cases}  x &\qquad  0\leq x\leq 1\\
 2-x&\qquad 1<x\leq 2 \\
 0  &< 2\leq x\leq 3 .
 \end{cases}\end{equation}

 The solution (\ref{DVinftyE3}) of Problem  (\ref{eq:3DV})  differs from the solution (\ref{DinftyE3}) of  problem  (\ref{eq:3DNO})  with homogenous Dirichlet data, moreover  $U \neq  u_{\infty, D}\vee 0.$

\begin {remark} \label{E3}  In Example 3, the datum  $f$   changes sign in $\Omega$  and it  is equal to $0$ in a set of positive measure.
All assumptions of Theorem \ref{main} are satisfied, in particular $\Gamma_p=[3-c_p,3],$\,\,\,  $c_p\to1^-$  and  $\lim \Gamma_p=[2,3]=\Gamma_\infty.$
 We note that also assumptions   (\ref{f2}) and  (\ref{f3})  are satisfied. 
 As we cannot use comparison principles (see \cite{LW}), then  we do not know whether the viscosity solution to problem (\ref{eq:3DV}) is unique: in any case  Theorems \ref{main} and \ref{tf}   select the variational solution $U,$ limit of the  functions $u_p.$   
 \end {remark}

\begin{figure}
\begin{center}
\includegraphics[width=4.5cm]{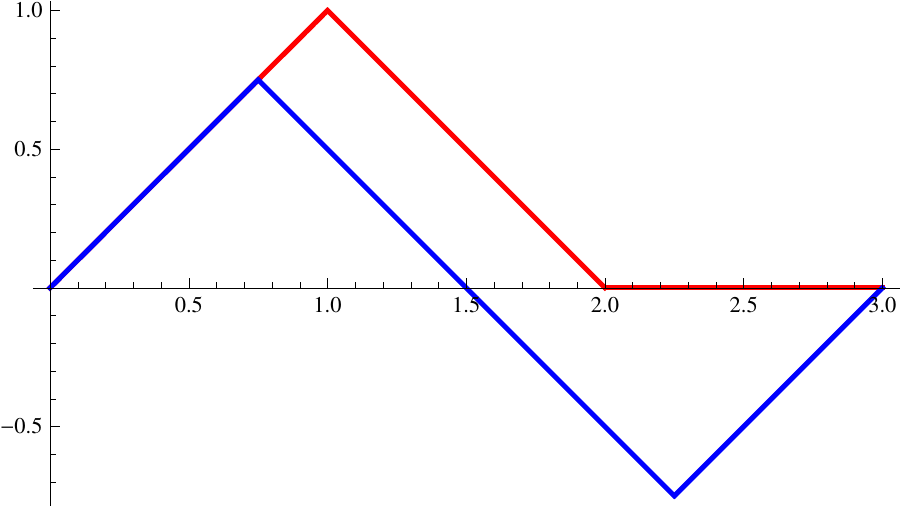}
\end{center}
\caption{ Example 3: $U$ in red, $u_{\infty, D}$ in blue.  }
\label{Figure3}
\end{figure}

  \bigskip
  
\textbf{Example 4}

Let $f= \chi_{(-2,-\frac32)\bigcup(\frac32,2) },$ $\Omega=(-2,2).$

  Now we consider the obstacle  $\varphi=1-x^2 .$ The solutions to the variational  inequality \eqref{OPp} is

 \begin{equation}\label{DV1E 4} u_p=
 \begin{cases} 
 1- x^2 &\qquad  |x|\leq c_p\\
-2c_p |x| +1 +c_p^2  &\qquad   c_p<|x| \leq  \frac{3}{2}\\
  \frac{-( |x| - \frac32  + (2c_p)^{p-1} )^{\beta +1}  + ( 2- \frac32  + (2c_p)^{p-1} )^{\beta +1}} {\beta +1}
 &\qquad   \frac{3}{2}<|x| \leq 2
   \end{cases}
   \end{equation}

where $$\beta=\frac1{p-1} $$
and \begin{equation}\label{cp5} c^2_p+1-3c_p =
 \frac{ (\frac12  +(2c_p)^{p-1})^{\beta+1}- ( 2c_p)^p }{\beta +1} .
\end{equation}

As $p\to \infty,$
from \eqref{cp5}, we obtain that  $c_p\to \frac{3-\sqrt{7}}{2}$   and the  limit of functions $u_p$  is

 \begin{equation}\label{DVinftyE4} U=\begin{cases}1-x^2  &\qquad  |x|\leq \frac{3-\sqrt{7}}{2}\\
 -(3-\sqrt{7})|x|+1+\frac{8-3\sqrt{7}}{2} &\qquad \frac{3-\sqrt{7}}{2}< |x|\leq \frac32\\
2- |x| &\qquad  \frac32<|x|\leq 2. 
 \end{cases}\end{equation} 

The solution  $u_{\infty, D}$  of  Problem  (\ref{eq:3DNO})  with homogenous Dirichlet data is 

\begin{equation}\label{DVinftyE40} u_{\infty, D}=\begin{cases} 2-|x|  &\qquad  \frac32<|x| \leq 2\\
  \frac12    &\qquad  |x|\leq \frac32.  \end{cases}\end{equation}

\begin {remark}\label{E4}
In  this example  all assumptions of Theorem  \ref{tf3} are satisfied  and the function $U$ in (\ref{DVinftyE4})  is a solution to Problem  (\ref{eq:3DV}) and  differs from the function $u_{\infty, D}$  in  (\ref{DVinftyE40}) that is a solution to  Problem  (\ref{eq:3DNO})  (with homogenous Dirichlet conditions), moreover  $U \neq  u_{\infty, D}\vee \varphi.$
Hence Example 4  shows that assumptions  of Theorem \ref{tf3} do not imply that the limit of the  solutions to  Problems (\ref{Dpn}) solves also Problem (\ref{eq:3DV}): in particular,  Theorem  \ref{tf3} is not a easy consequence  of  Theorem 2.4  in \cite{HL}.

\end {remark}

\begin{figure}
\begin{center}
\includegraphics[width=4.5cm]{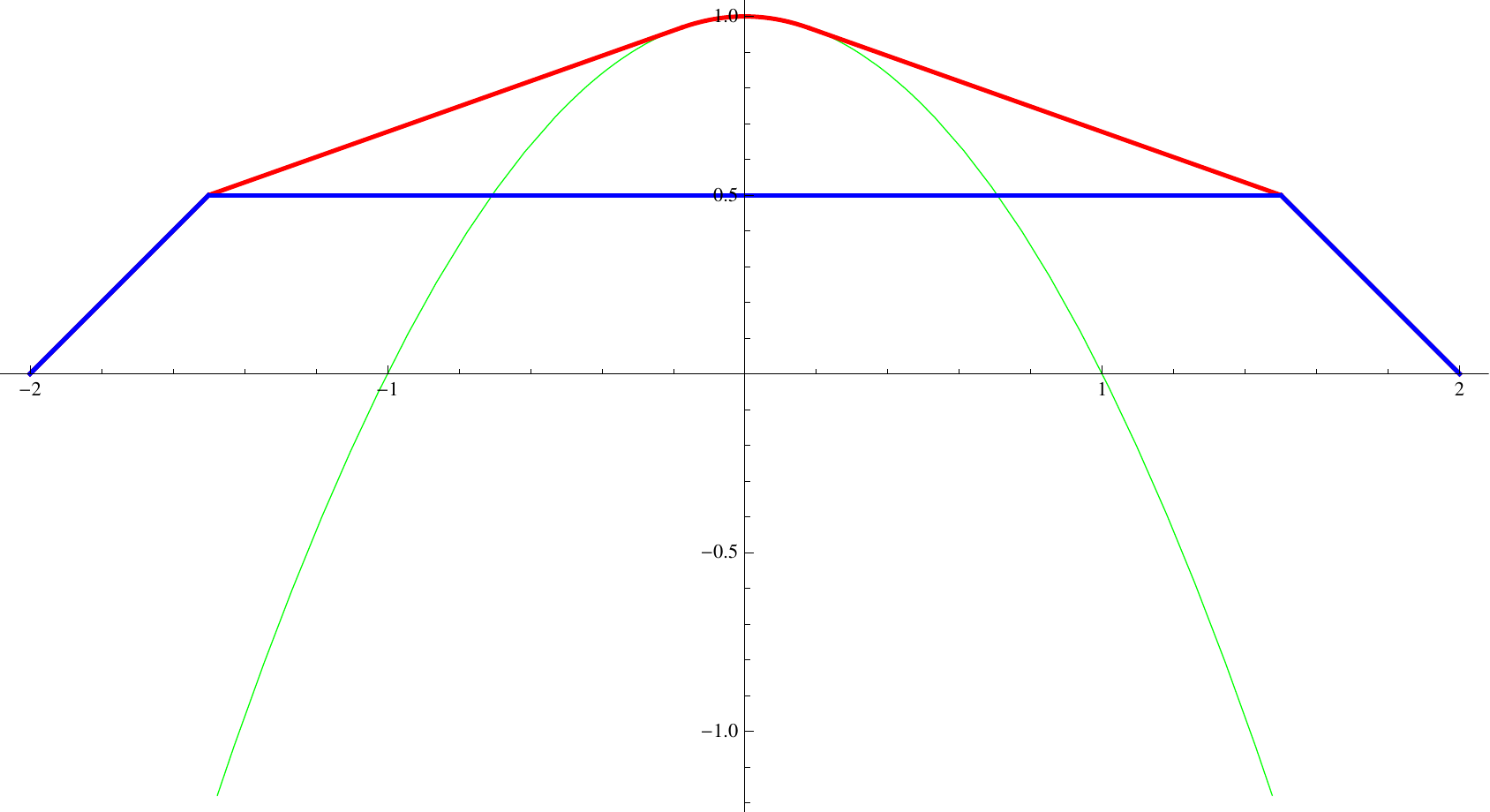}\
\end{center}
\caption{ Example 4: $U$ in red, $u_{\infty, D}$ in blue,  obstacle in green. }
\label{Figure4}
\end{figure}

\textbf{{Example 5}}

 Let $f=-\chi_{(-\frac13,\frac13)},$ $\Omega=(-1,1).$
  Now we consider the obstacle  $\varphi=\frac{3}{4}(x^2-1).$ The solution to the variational  inequality is

 \begin{equation}\label{DV1E 5} u_p=\begin{cases}  \frac{3}{4}(x^2-1) & \qquad  |x|\leq c_p\\
   \frac{ (|x|+ (\frac32 c_p)^{p-1}- c_p)^{\beta +1}-(\frac13+ (\frac32 c_p)^{p-1}- c_p)^{\beta}   (1+ (\frac32 c_p)^{p-1}- c_p+ \frac{2\beta} 3) }{\beta +1}   &\qquad  c_p<|x| \leq  \frac{1}{3} \\
(\frac13+ (\frac32 c_p)^{p-1}- c_p)^{\beta}(|x| -1)  &\qquad \frac13<|x|\leq1
  \ \end{cases}\end{equation}
where $$\beta=\frac1{p-1} $$
and \begin{equation}\label{cp6} \frac34 (c^2_p-1)=  \frac{ (\frac32 c_p)^{\frac{\beta +1}\beta}-(\frac13+ (\frac32 c_p)^{p-1}- c_p)^{\beta}   (1+ (\frac32 c_p)^{p-1}- c_p+ \frac{2\beta} 3) }{\beta +1} .\end{equation}

 As $p\to \infty,$
from \eqref{cp6}, we obtain that  $c_p\to \frac13,$    $(\frac13+ (\frac32 c_p)^{p-1}- c_p)^{\beta}\to 1,$    $\Gamma_p=[-c_p, c_p],$   $\lim \Gamma_p=[-\frac13, \frac13]$    and the  limit of functions $u_p$  is 
 \begin{equation}\label{DVinftyE5} U=\begin{cases} \frac34(x^2-1) &\qquad  |x|\leq \frac13\\
 |x|-1 &\qquad \frac13<|x|\leq 1. 
 \end{cases}\end{equation}

In this example, 
 the solution  $U$ in (\ref{DVinftyE5}) of Problem  (\ref{eq:3DV}) differs from the function  $u_{\infty, D}= -d(x)$  solution to  Problem  (\ref{eq:3DNO})  with homogenous  Dirichlet data.

 \begin {remark} \label{E5}  In Example 5  the assumptions  of Theorem \ref{main}  are satisfied,  nevertheless   the limit of the solutions to   Problems (\ref{Dpn})  does not solve Problem (\ref{eq:3DV}). In particular  Theorem \ref{main} is not a easy consequence  of  Theorem 2.1  in \cite {HL}.
\end {remark}

\begin{figure}
\begin{center}
\includegraphics[width=4.5cm]{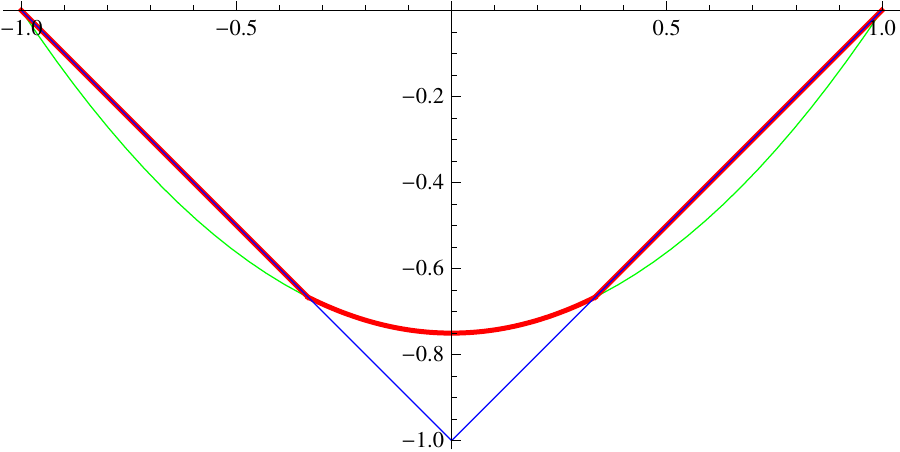}\
\end{center}
\caption{ Example 5: $U$ in red, $u_{\infty, D}$ in blue, obstacle in green.  }
\label{Figure5}
\end{figure}

\textbf{{Example 6}}

Let  $\Omega=\{x\in \mathbb{R}^n: |x|<2\},\,\,$  $\varphi= 1-x^2$  and $f=0$ (see example in the Appendix of \cite{RTU}).

Problem (\ref{eq:3DV}) does not have a unique solution; in fact both the  following functions  $U$ and $v^*$ are solutions: 
\begin{equation}\label{uRTU}
 U=\begin{cases}  
1-x^2&\qquad |x|\leq h \\
-2h|x|+4h &\qquad   h< |x|\leq 2
 \end{cases} 
\end{equation}
where $h=2-\sqrt{3}$
and
\begin{equation}\label{vRTU}
 v^*=\begin{cases}  
1-x^2&\qquad |x|\leq \frac12 \\
-|x|+\frac54&\qquad
\frac12  \leq|x|\leq\frac54\\
0   &\quad\quad         \frac54\   <|x|\leq2.
\end{cases} 
\end{equation}

For $p>n$ and $\alpha=\frac{n-1}{p-1}$, we have that  the solution to \eqref{OPp} is
\begin{equation}\label{upRTU}
u_p=\begin{cases}  
1-x^2\qquad |x|\leq c_p \\
\frac{-2 c_p^{1+\alpha}}{1-\alpha} (  |x|^{1-\alpha}-2^{1-\alpha})\qquad   c_p< |x|\leq 2
\end{cases} 
\end{equation}
where 

\begin{equation}\label{cp7} (1+\alpha) c_p^2-2^{2-\alpha}  c_p^{1+\alpha}+ 1-\alpha =0. \end{equation} 

 As $p\to\infty,$ from \eqref{cp7}, we obtain that  $c_p\to h$ and    $\lim\Gamma_p=\lim [-c_p,c_p]=[-h, h]=\Gamma_\infty.$
  The function $U,$ that  (according Remark \ref{cerchio} of  Theorem \ref{radial}) is  limit of $u_p,$  coincides on the annulus $B_{h,2}$  with  the AMLE of $g$ ,
\begin{equation}\label{g}
g=\begin{cases}  
1-h^2&\qquad x\in\partial B_h \\
0 &\qquad  x\in\partial B_2
\end{cases} 
\end{equation}
 while the function $v^*$ is  a  solution  of Problem (\ref{eq:3DV}), but it is not  the AMLE of  g.

\begin{figure}
\begin{center}
\includegraphics[width=4.5cm]{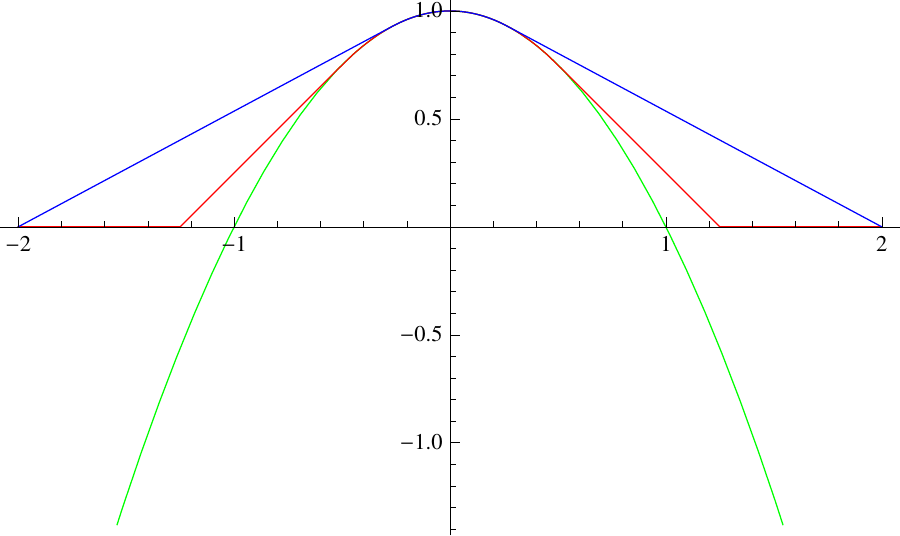}\
\end{center}
\caption{ Example 6: $U$ in red, $v^*$ in blue, obstacle in green.  }
\label{Figure6}
\end{figure}

\textbf{{Example 7}}

Let  $\Omega=\{x\in \mathbb{R}^n: |x|<2\},\,\,$  $\varphi= 1-x^2$  and  $f=-1.$

The  function $U,$ limit of $u_p$  coincides  with the unique viscosity solution to Problem (\ref{eq:3DV}), $u_\infty$  (see Theorem \ref{tf2}). More precisely

\begin{equation}\label{u_pRTU}
u_p=\begin{cases} 
  1-x^2 &\qquad |x|\leq h_p \\
 \frac{(-|x|+c_p)}{\beta+1} ^{\beta+1}- \frac{(-c_p+2)}{\beta+1} ^{\beta+1} 
&\qquad h_p<|x|\leq  c_p \\
 \frac{(|x|-c_p)}{\beta+1} ^{\beta+1}- \frac{(-c_p+2)}{\beta+1} ^{\beta+1}  &\qquad c_p \leq |x|\leq  2 \\ \\
 \end{cases}
\end{equation}

where $$\beta=\frac1{p-1}, \,\,\,(-h_p+ c_p)^\beta=2h_p,\,\,(-h_p+ c_p)^{\beta+1} - (2- c_p)^{\beta+1}=(\beta+1)(1-h^2_p).$$
Then $h_p\to \frac{1}{2},\,\,\,c_p \to \frac{13}{8}$ and the limit function is 

\begin{equation}\label{URTU}
U=\begin{cases} 1-x^2 \qquad  |x|\leq \frac12 \\
\frac{5}{4}   -|x|  \qquad  \frac12<|x|\leq  \frac{13}{8}\\
|x|-2  \qquad \frac{13}{8}<|x|\leq 2,
\end{cases} 
\end{equation}
while the function $u_{\infty, D}$ coincides with the opposite of the distance from the boundary, $u_{\infty, D}=-d(x)= 2-|x|.$ 

\bigskip

\textbf{GRANTS } The authors are members of GNAMPA (INdAM) and are partially supported by Grant
Ateneo \lq\lq Sapienza" 2017.

\end{document}